\documentclass[12pt]{article}
\usepackage{amsfonts}
\usepackage{amsmath,color,diagrams}
\usepackage{hyperref}
\usepackage[makeroom]{cancel}

\newcommand{\beq}{\begin{equation}}
\newcommand{\eeq}{\end{equation}}
\newcommand{\bea}{\begin{eqnarray}}
\newcommand{\eea}{\end{eqnarray}}
\newcommand{\beas}{\begin{eqnarray*}}
\newcommand{\eeas}{\end{eqnarray*}}

\newtheorem{theorem}{Theorem}[section]

\newtheorem{definition}[theorem]{Definition}
\newtheorem{proposition}[theorem]{Proposition}
\newtheorem{prop}[theorem]{Proposition}
\newtheorem{corollary}[theorem]{Corollary}
\newtheorem{lemma}[theorem]{Lemma}
\newtheorem{remark}[theorem]{Remark}
\newtheorem{example}[theorem]{Example} 
\newtheorem{examples}[theorem]{Examples}
\newtheorem{foo}[theorem]{Remarks}

\newenvironment{Example}{\begin{example}\rm}{\end{example}}

\newenvironment{proof}{\addvspace{\medskipamount}\par\noindent{\it
Proof}.}
{\unskip\nobreak\hfill$\Box$\par\addvspace{\medskipamount}}

\newcommand{\p}{\partial}
\newcommand{\ee}{\ell}

\newcommand{\bM}{\mathbb M}

\newcommand{\rL}{\overline{\Delta}_{\mathcal{H}} }
\newcommand{\Dh}{{\Delta}_{\mathcal{H}}}
\newcommand{\Dv}{{\Delta}_{\mathcal{V}}}
\newcommand{\Ho}{\mathcal H}
\newcommand{\V}{\mathcal V}

\newcommand{\M}{\mathbb M}
\newcommand{\B}{\mathbb B}

\newcommand{\rp}{\overline p}
\newcommand{\bS}{\mathbb S}

\newcommand{\R}{\mathbb R}

\newcommand{\ve}{\varepsilon}

\newcommand{\Ric}{\mathfrak Ric}
\newcommand{\ch}{\mathcal H}
\newcommand{\J}{\mathfrak J}

\newcommand{\ep}{\varepsilon}
\title{Sub-Laplacians and hypoelliptic operators on totally geodesic Riemannian foliations}
\author{Fabrice Baudoin\footnote{fbaudoin@purdue.edu}}

\date{Department of Mathematics, Purdue University}

\begin{document}
\maketitle

\begin{abstract}
 These notes are the basis of a course given at the Institut Henri Poincar\'e in September 2014. We survey some recent results related to the geometric analysis of hypoelliptic diffusion operators on totally geodesic Riemannian foliations. We also give new applications to the study of hypocoercive estimates for Kolmogorov type operators.
\end{abstract}

\tableofcontents

\section{Introduction}
 
It is a fact that many interesting hypoelliptic diffusion operators may be studied by introducing a well-chosen Riemannian foliation. In particular, several sub-Laplacians on  sub-Riemannians manifolds often appear as  horizontal Laplacians of a foliation and several of the Kolmogorov type hypoelliptic diffusion operators which are used in the theory of kinetic  equations appear as the sum of the vertical Laplacian of a foliation and of a first order term.  

\

The  goal of the present notes is to survey some geometric analysis tools to study this kind of diffusion operators. We specially would like to stress the importance of subelliptic Bochner's type identities in this framework and show how they can be used to deduce a variety of results ranging from topological informations on a sub-Riemannian manifold to hypocoercive estimates and convergence to equilibrium for kinetic Fokker-Planck equations. As an illustration of those methods we give a proof of a  sub-Riemannian Bonnet-Myers type compactness theorem (Section 6) and study a version of the Bakry-\'Emery criterion for Kolmogorov type operators (Section 7).

\

For  the proof of the sub-Riemannian Bonnet-Myers theorem we adapt an approach developed in a joint program with Nicola Garofalo. The object of this program initiated in \cite{BG0,BG} has been to propose a generalized curvature dimension inequality that fits a number of interesting subelliptic situations including the ones considered in these notes. While some of them will be discussed here, the numerous applications of the generalized curvature dimension inequality  are beyond the scope of these notes and we will only give the relevant pointers to the literature.  We focus here  more on the Bonnet-Myers theorem and the geometric framework in which this  curvature-dimension estimate is available.

\

Concerning Section 7 most of the material is actually new,  though the main ideas originate from \cite{baudoin-bakry}.

\

These notes are organized as follows. 

\

\textbf{Section 2:}  We  introduce the concept of Riemannian foliation and define the horizontal and vertical Laplacians. Basic theorems like the B\'erard-Bergery-Bourguignon commutation theorem will be proved.

\

\textbf{Section 3:} We study  in details some examples of Riemannian foliations with totally geodesic leaves that can be seen as model spaces. Besides the Heisenberg group, these examples are  associated to the Hopf fibrations  on the sphere. We  give explicit expressions for the radial parts of the horizontal and vertical Laplacians and for the horizontal heat kernels of these model spaces.

\

\textbf{Section 4:} We prove a transverse Weitzenb\"ock formula for the horizontal Laplacian of a Riemannian foliation with totally geodesic leaves. It is the main geometric analysis tool for the study of the horizontal Laplacian. As a first consequence of this Weitzenb\"ock formula, we prove that if natural assumptions are satisfied, then the horizontal Laplacian satisfies the generalized curvature dimension inequality. As a second consequence, we will prove sharp lower bounds for the first eigenvalue of the horizontal Laplacian.

\

\textbf{Section 5:} In this section, we introduce the horizontal semigroup of a Riemannian foliation with totally geodesic leaves and discuss fundamental questions like essential self-adjointness for the horizontal Laplacian and stochastic completeness. We also prove Li-Yau gradient bounds for this horizontal semigroup.

\

\textbf{Section 6:} By using semigroup methods, we prove  a sub-Riemannian Bonnet-Myers  theorem in the context of  Riemannian foliations with totally geodesic leaves. 

\

\textbf{Section 7:} This last section is an introduction to the analysis of hypoelliptic Kolmogorov type operators on Riemannian foliations. We  mainly focus on the problem of convergence to equilibrium  for the parabolic equation associated to the operator and on methods to prove hypocoercive estimates. The example of the kinetic Fokker-Planck equation is given as an illustration.

\section{Riemannian foliations and their Laplacians}

We review first some basic facts about the geometry of Riemannian foliations that will be needed in the sequel. In particular, we define the horizontal and vertical Laplacians on such foliations and show that they commute if the metric is bundle like and the foliation totally geodesic. For further details about the geometry of Riemannian submersions we refer to Chapter 9 in \cite{Besse} and for more informations  about general Riemannian foliations, we refer to the book by Tondeur \cite{Tondeur}.

\subsection{Riemannian submersions}

Let $(\M , g)$  and  $(\B,j)$ be  smooth and connected Riemannian manifolds.
\begin{definition}
A smooth surjective map $\pi: (\M , g)\to (\B,j)$ is called a Riemannian submersion if its derivative maps $T_x\pi : T_x \M \to T_{\pi(x)} \B$ are orthogonal projections, i.e. for every $ x \in \M$, the map $ T_{x} \pi (T_{x} \pi)^*: T_{p(x)}  \B \to T_{p(x)} \B$ is the identity.
\end{definition}

\begin{Example}(\textbf{Warped products}) Let $(\M_1 , g_1)$  and  $(\M_2,g_2)$ be   Riemannian manifolds and $f$ be a smooth and positive function on $\M_1$. Then the first projection $(\M_1 \times \M_2,g_1 \oplus f g_2) \to (\M_1, g_1)$ is a Riemannian submersion.
\end{Example}

\begin{Example}(\textbf{Quotient by an isometric action})
Let $(\M , g)$ be a Riemannian manifold and $\mathbb G$ be a closed subgroup of the isometry group of $(\M , g)$. Assume that the projection map $\pi$ from $\M$ to the quotient space $\M /\mathbb{G}$ is a smooth submersion. Then there exists a unique Riemannian metric $j$ on $\M /\mathbb{G}$ such that $\pi$ is a Riemannian submersion.
\end{Example}

If $\pi$ is a Riemannian submersion and $b \in \B$, the set  $\pi^{-1}(\{ b \})$ is called a fiber. 

\

For $ x \in \M$, $\mathcal{V}_x =\mathbf{Ker} (T_x\pi)$ is called the vertical space at $x$. The orthogonal complement of $\mathcal{H}_x$ shall be denoted $\mathcal{H}_x$ and will be referred to as the horizontal space at $x$. We have an orthogonal decomposition
\[
T_x \M=\mathcal{H}_x \oplus \mathcal{V}_x
\]
and a corresponding splitting of the metric
\[
g=g_{\mathcal{H}} \oplus g_{\mathcal{V}}.
\]
The vertical distribution $\mathcal V$ is of course integrable since it is the tangent distribution to the fibers, but the horizontal distribution is in general not integrable. Actually, in all the situations we will consider  the horizontal distribution is everywhere bracket-generating in the sense that for every $x \in \M$, $\mathbf{Lie} (\mathcal{H}) (x)=T_x \M$. In that case it is natural to study the sub-Riemannian geometry of the triple $(\M, \mathcal{H}, g_{\mathcal{H}})$. As we will see, many interesting examples of sub-Riemannian structures arise in this framework and this is really the situation which is interesting for us.

\

We shall mainly be interested in submersion with totally geodesic fibers.

\begin{definition}
A Riemannian submersion $\pi: (\M , g)\to (\B,j)$ is said to have totally geodesic fibers if for every $b \in \B$, the set $\pi^{-1}(\{ b \})$ is a totally geodesic submanifold of $\M$.
\end{definition}

\begin{Example}(\textbf{Quotient by an isometric action})
Let $(\M , g)$ be a Riemannian manifold and $\mathbb G$ be a closed one-dimensional subgroup of the isometry group of $(\M , g)$ which is generated by a complete Killing vector field $X$. Assume that the projection map $\pi$ from $\M$ to $\M /\mathbb{G}$ is a smooth submersion.  Then the fibers are totally geodesic if and only if the integral curves of $X$ are geodesics, which is the case if and only if $X$ has a constant length.
\end{Example}

\begin{Example}(\textbf{Principal bundle})
Let $\M$ be a principal bundle over $\B$ with fiber $\mathbf F$ and structure group $\mathbb G$. Then, given a Riemannian metric $j$ on $\B$, a $\mathbb G$-invariant  metric $k$ on $ \mathbf F$ and a $\mathbb G$ connection form $\theta$, there exists a unique Riemannian metric $g$ on $\M$ such that the bundle projection map $\pi: \M \to \B$ is a Riemannian submersion with totally geodesic fibers isometric to $(\mathbf{F},k)$ and such that the horizontal distribution of $\theta$ is the orthogonal complement of the vertical distribution. We refer to \cite{Vilms}, page 78, for a proof. In the case of the tangent bundle of a Riemannian manifold, the construction yields the Sasaki metric on the tangent bundle.
\end{Example}

As we will see, for a Riemannian submersion with totally geodesic fibers, all the fibers are isometric. The argument, due to Hermann \cite{Hermann} relies on the notion of basic vector field that we now introduce.

\

Let $\pi: (\M , g)\to (\B,j)$ be a Riemannian submersion. A vector field $X \in \Gamma^\infty(T\M)$ is said to be projectable if there exists a smooth vector field $\overline{X}$ on $\B$ such that for every $x \in \M$,  $T_x \pi ( X(x))= \overline {X} (\pi (x))$. In that case, we say that $X$ and $\overline{X}$ are $\pi$-related.

\begin{definition}
A vector field $X$ on $\M$ is called basic if it is projectable and horizontal.
\end{definition}

If  $\overline{X}$ is a smooth vector field on $\B$, then there exists a unique basic vector field $X$ on $\M$ which is $\pi$-related to $\overline{X}$. This vector is called the lift of $\overline{X}$.

\

Notice that if $X$ is a basic vector field and $Z$ is a vertical vector field, then $T_x\pi ( [X,Z](x))=0$ and thus $[X,Z]$ is a vertical vector field. The following result is due to Hermann \cite{Hermann}.

\begin{proposition}\label{isometry}
The submersion $\pi$ has totally geodesic fibers if and only if the flow generated by any basic vector field induces an isometry between the fibers.
\end{proposition}

\begin{proof}
We denote by $D$ the Levi-Civita connection on $\M$.  Let $X$ be a basic vector field. If $Z_1,Z_2$ are vertical fields, the Lie derivative of $g$ with respect to $X$ can be computed as
\[
(\mathcal{L}_X g)(Z_1,Z_2)=\langle D_{Z_1} X ,Z_2 \rangle +\langle D_{Z_2} X ,Z_1 \rangle.
\]
Because $X$ is orthogonal to $Z_2$, we now have $\langle D_{Z_1} X ,Z_2 \rangle=-\langle  X ,D_{Z_1} Z_2 \rangle$. Similarly $\langle D_{Z_2} X ,Z_1 \rangle=-\langle  X ,D_{Z_2} Z_1 \rangle$. We deduce
\begin{align*}
(\mathcal{L}_X g)(Z_1,Z_2)& =-\langle  X ,D_{Z_1} Z_2 +D_{Z_2} Z_1 \rangle \\
 &=-2 \langle  X ,D_{Z_1} Z_2 \rangle.
\end{align*}
Thus the flow generated by any basic vector field induces an isometry between the fibers if and only if $D_{Z_1} Z_2$ is always vertical which is equivalent to the fact that the fibers are totally geodesic submanifolds.
\end{proof}

\subsection{The horizontal and vertical Laplacians}

Let $\pi: (\M , g)\to (\B,j)$ be a Riemannian submersion.  If $f \in C^\infty(\M)$ we define its vertical gradient $\nabla_{\mathcal{V}}$ as the projection of its gradient onto the vertical distribution and its horizontal gradient $\nabla_{\mathcal{H}}$ as the projection of the gradient onto the horizontal distribution. We define then the vertical Laplacian  $\Delta_{\mathcal{V}}$ as the generator of the Dirichlet form
\[
\mathcal{E}_{\mathcal{V}}(f,g)=-\int_\M \langle \nabla_{\mathcal{V}} f , \nabla_{\mathcal{V}} g \rangle d\mu,
\]
where $\mu$ is the Riemannian volume measure on $\M$. Similarly, we define  the horizonal Laplacian  $\Delta_{\mathcal{H}}$ as the generator of the Dirichlet form
\[
\mathcal{E}_{\mathcal{H}}(f,g)=-\int_\M \langle \nabla_{\mathcal{H}} f , \nabla_{\mathcal{V}} g \rangle d\mu.
\]
If $X_1,\cdots,X_n$ is a local orthonormal frame of basic vector fields and $Z_1,\cdots,Z_m$ a local orthonormal frame of the vertical distribution, then we have
\[
\Delta_{\mathcal{H}}=-\sum_{i=1}^n X_i^* X_i
\]
and
\[
\Delta_{\mathcal{V}}=-\sum_{i=1}^m Z_i^* Z_i,
\]
where the adjoints are understood in $L^2(\mu)$. Classically, we have
\[
X_i^*=-X_i+\sum_{k=1}^n \langle D_{X_k} X_k, X_i\rangle +\sum_{k=1}^m \langle D_{Z_k} Z_k, X_i\rangle,
\]
where $D$ is the Levi-Civita connection. As a consequence, we obtain
\[
\Delta_{\mathcal{H}}=\sum_{i=1}^n X_i^2 -\sum_{i=1}^n (D_{X_i}X_i)_{\mathcal{H}} -\sum_{i=1}^m (D_{Z_i}Z_i)_{\mathcal{H}},
\]
where $(\cdot)_{\mathcal{H}}$ denotes the horizontal part of the vector. In a similar way we obviously have
\[
\Delta_{\mathcal{V}}=\sum_{i=1}^m Z_i^2 -\sum_{i=1}^n (D_{X_i}X_i)_{\mathcal{V}} -\sum_{i=1}^m (D_{Z_i}Z_i)_{\mathcal{V}}.
\]
We can observe that the Laplace-Beltrami operator $\Delta$ of $\M$ can be written
\[
\Delta=\Delta_{\mathcal{H}}+\Delta_{\mathcal{V}}.
\]
It is worth noting that, in general, $\Delta_{\mathcal{H}}$ is not the lift of the Laplace-Beltrami operator $\Delta_\mathbb{B}$ on  $\B$. Indeed, let us denote by $\overline{X}_1,\cdots,\overline{X}_n$ the vector fields on $\B$ which are $\pi$-related to $X_1,\cdots,X_n$ . We have
\[
\Delta_{\B}=\sum_{i=1}^n \overline{X}_i^2 -\sum_{i=1}^n D_{\overline{X}_i}\overline{X}_i.
\]
Since it is easy to check that  $ D_{\overline{X}_i}\overline{X}_i$ is $\pi$-related to $(D_{X_i}X_i)_{\mathcal{H}}$, we deduce that $\Delta_{\mathcal{H}}$ lies above $\Delta_{\B}$, i.e. for every $f \in C^\infty(\B)$, $\Delta_{\mathcal{H}} (f \circ \pi)=(\Delta_{\B} f )\circ \pi$ , if and only if the vector
\[
T=\sum_{i=1}^m D_{Z_i}Z_i
\]
is vertical. This condition is equivalent to the fact that the mean curvature of each fiber is zero, or in other words that the fibers are minimal submanifolds of $\M$. This happens for instance for submersions with totally geodesic fibers.

\

We also note that from H\"ormander's theorem, the operator $\Delta_{\mathcal{H}}$ is subelliptic if the horizontal distribution is bracket generating. Of course, the vertical Laplacian is never subelliptic because the vertical distribution is always integrable.

\

The following result, though simple, will turn out  to be extremely useful in the sequel when dealing with curvature dimension estimates and functional inequalities.

\begin{theorem}\label{inter_commutation}
The Riemannian submersion $\pi$ has totally geodesic fibers if and only if for every $f \in C^\infty(\M)$,
\[
\langle \nabla_{\mathcal{H}} f , \nabla_{\mathcal{H}} \| \nabla_{\mathcal{V}} f \|^2 \rangle=\langle \nabla_{\mathcal{V}} f , \nabla_{\mathcal{V}} \| \nabla_{\mathcal{H}} f \|^2 \rangle
\]
\end{theorem}
 
 \begin{proof}
 If $X_1,\cdots,X_n$ is a local orthonormal frame of basic vector fields and $Z_1,\cdots,Z_m$ a local orthonormal frame of the vertical distribution, then we easily compute that
 \[
 \langle \nabla_{\mathcal{H}} f , \nabla_{\mathcal{H}} \| \nabla_{\mathcal{V}} f \|^2 \rangle-\langle \nabla_{\mathcal{V}} f , \nabla_{\mathcal{V}} \| \nabla_{\mathcal{H}} f \|^2 \rangle=2\sum_{i=1}^n \sum_{j=1}^m (X_i f) (Z_j f) ([X_i,Z_j] f).
 \]
 As a consequence, 
 \[
\langle \nabla_{\mathcal{H}} f , \nabla_{\mathcal{H}} \| \nabla_{\mathcal{V}} f \|^2 \rangle=\langle \nabla_{\mathcal{V}} f , \nabla_{\mathcal{V}} \| \nabla_{\mathcal{H}} f \|^2 \rangle
\]
if and only if for every basic vector field $X$,
\[
 \sum_{j=1}^m  (Z_j f) ([X,Z_j] f)=0.
 \]
This condition is equivalent to the fact that the flow generated by  $X$ induces an isometry between the fibers, and so from Hermann's Theorem \ref{isometry} this equivalent to the  fact that the fibers are totally geodesic.
 \end{proof}
 
 The second commutation result  that characterizes totally geodesic submersions is due to B\'erard-Bergery and Bourguignon \cite{BeBo}.
 
 \begin{theorem}\label{commutation2}
The Riemannian submersion $\pi$ has totally geodesic fibers if and only if any basic vector field $X$ commutes with the vertical Laplacian $\Delta_{\mathcal{V}}$. In particular, if $\pi$ has totally geodesic fibers, then for every $f \in C^\infty(\M)$,
\[
\Delta_{\mathcal{H}} \Delta_{\mathcal{V}} f=\Delta_{\mathcal{V}} \Delta_{\mathcal{H}} f.
\]
\end{theorem}
 
 \begin{proof}
 Assume that the submersion is totally geodesic. Let $X$ be a basic vector field and $\xi_t$ be the flow it generates. Since $\xi$ induces an isometry between the fibers, we have
 \[
 \xi_t^* ( \Delta_{\mathcal{V}})= \Delta_{\mathcal{V}}.
 \]
 Differentiating at $t=0$ yields $[X,\Delta_{\mathcal{V}}]=0$. 
 
 Conversely, assume that for every basic field $X$, $[X,\Delta_{\mathcal{V}}]=0$.  Let $X_1,\cdots,X_n$ be a local orthonormal frame of basic vector fields and $Z_1,\cdots,Z_m$ be a local orthonormal frame of the vertical distribution. The second order part of the operator $[X,\Delta_{\mathcal{V}}]$ must be zero. Given the expression of $ \Delta_{\mathcal{V}}$, this implies
 \[
 \sum_{i=1}^m [X,Z_i] Z_i=0.
 \]
 So $X$ leaves the symbol of $\Delta_{\mathcal{V}}$ invariant which is the metric on the vertical distribution. This implies that the flow generated by $X$ induces isometries between the fibers.
 
 \
 
 Finally, as we have seen, if the submersion is totally geodesic then in a local basic orthornomal frame
 \[
 \Delta_{\mathcal{H}}=\sum_{i=1}^n X_i^2 -\sum_{i=1}^n (D_{X_i}X_i)_{\mathcal{H}}.
 \]
 Since the vectors  $(D_{X_i}X_i)_{\mathcal{H}}$ are basic,  from the previous result $\Delta_{\mathcal{H}}$ commutes with $\Delta_{\mathcal{V}}$.
  \end{proof}
 
\subsection{Riemannian foliations}

In many interesting cases, we do not actually have a globally defined Riemannian sumersion but a Riemannian foliation.

\begin{definition}
Let $\M$ be a smooth and connected $n+m$ dimensional manifold. A $m$-dimensional foliation $\mathcal{F}$ on $\M$ is defined by a maximal collection of pairs $\{ (U_\alpha, \pi_\alpha), \alpha \in I \}$ of open subsets $U_\alpha$ of $\M$ and submersions $\pi_\alpha: U_\alpha \to U_\alpha^0$ onto open subsets of $\mathbb{R}^n$ satisfying:
\begin{itemize}
\item $\cup_{\alpha \in I} U_\alpha =\M$;
\item If $U_\alpha \cap U_\beta \neq \emptyset$, there exists a local diffeomorphism $\Psi_{\alpha \beta}$ of $\mathbb{R}^n$ such that $\pi_\alpha=\Psi_{\alpha \beta} \pi_\beta$ on $U_\alpha \cap U_\beta $.
\end{itemize}
\end{definition}

The maps $\pi_\alpha$ are called disintegrating maps of $\mathcal{F}$. The connected components of the sets $\pi_\alpha^{-1}(c)$, $c \in \mathbb{R}^n$, are called the plaques of the foliation. A foliation arises from an integrable sub-bundle of $T\M$, to be denoted by $\mathcal{V}$ and referred to as the vertical distribution. These are the vectors tangent to the leaves, the maximal integral sub-manifolds of $\mathcal{V}$.  

\

Foliations have been extensively studied and numerous books are devoted to them. We refer in particular to the book by Tondeur \cite{Tondeur}.

\

In the sequel, we shall only  be interested in Riemannian foliations with bundle like metric. 

\begin{definition}
Let $\M$ be a smooth and connected $n+m$ dimensional Riemannian manifold. A $m$-dimensional foliation $\mathcal{F}$ on $\M$ is said to be Riemannian with a bundle like metric if the disintegrating maps $\pi_\alpha$ are Riemannian submersions onto $U_\alpha^0$ with its given Riemannian structure. If moreover the leaves are totally geodesic sub-manifolds of $\M$, then we say that the Riemannian foliation is totally geodesic with a bundle like metric.
\end{definition}

Observe that if we have a Riemannian submersion $\pi : (\M,g) \to (\mathbb{B},j)$, then $\M$ is equipped with a Riemannian foliation with bundle like metric whose leaves are the fibers of the submersion. Of course, there are many Riemannian foliations with bundle like metric that do not come from a Riemannian submersion.

\begin{Example}(\textbf{Contact manifolds})
Let $(\M,\theta)$ be a $2n+1$-dimensional smooth contact manifold. On $\M$ there is a unique smooth vector field $T$, the so-called Reeb vector field, that satisfies
\[
\theta(T)=1,\quad \mathcal{L}_T(\theta)=0,
\]
where $\mathcal{L}_T$ denotes the Lie derivative with respect to  $T$. On $\M$ there is a foliation, the Reeb foliation, whose leaves are the orbits of the vector field $T$.  As it is well-known (see for instance \cite{Tanno}), it is always possible to find a Riemannian metric $g$ and a $(1,1)$-tensor field $J$ on $\M$ so that for every  vector fields $X, Y$
\[
g(X,T)=\theta(X),\quad J^2(X)=-X+\theta (X) T, \quad g(X,JY)=(d\theta)(X,Y).
\]
The triple $(\M, \theta,g)$ is called a contact Riemannian manifold. We see then that the Reeb foliation is totally geodesic with bundle like metric if and only if the Reeb vector field $T$ is a Killing field, that is,
\[
\mathcal{L}_T g=0.
\]
In that case $(\M, \theta,g)$ is called a K-contact Riemannian manifold.
\end{Example}

\begin{Example}(\textbf{Sub-Riemannian manifolds with transverse symmetries})
The concept of sub-Riemannian manifold with transverse symmetries was introduced in \cite{BG}. Let $\M$ be a smooth, connected  manifold with dimension $n+m$. We assume that $\bM$ is equipped with a bracket generating distribution $\mathcal{H}$ of dimension $n$ and a fiberwise inner product $g_\mathcal{H}$ on that distribution. It is said that $\M$ is a sub-Riemannian manifold with transverse symmetries if there exists a $m$- dimensional Lie algebra $\mathcal{V}$ of sub-Riemannian Killing vector fields such that for every $x \in \bM$, 
 \[
 T_x \bM= \mathcal{H}(x) \oplus \mathcal{V}(x),
 \]
 where 
 \[
  \mathcal{V}(x)=\{ Z(x), Z \in \mathcal{V}(x) \}.
 \]
 The choice of an inner product $g_{\mathcal{V}}$ on the Lie algebra $\mathcal{V}$ naturally endows $\bM$ with a  Riemannian metric that makes the decomposition $T_x\M=\mathcal{H}(x) \oplus \mathcal{V}(x)$ orthogonal:
\[
g=g_\mathcal{H} \oplus  g_{\mathcal{V}}.
\]
The sub-bundle of $\M$ determined by vector fields in $\mathcal{V}$ gives a foliation on $\M$ which is easily seen to be totally geodesic with bundle like metric.
\end{Example}

Since Riemannian foliations with a bundle like metric can locally be desribed by a Riemannian submersion, we can define a horizontal Laplacian $\Dh$ and a vertical Laplacian $\Dv$. Observe that they commute on smooth functions if the foliation is totally geodesic. More generally all the local properties of a Riemannian submersion extend to Riemannian foliations.

\section{Horizontal Laplacians and heat kernels on model spaces}

We  discuss concrete examples of Riemannian foliations with totally geodesic leaves and bundle like metric. We focus in particular on the study of the horizontal Laplacians and of the corresponding heat kernels for which we show that explicit expressions can be given. The examples we cover are the Heisenberg group and the Hopf fibrations on the sphere. They can respectively be seen as the models of \textit{flat} and \textit{positively curved} sub-Riemannian spaces. The \textit{negatively curved} sub-Riemannian spaces  come from totally geodesic pseudo-Riemannian foliations on the anti-de Sitter space and for more detais we refer to the thesis of Michel Bonnefont \cite{Bonnefont1} and Jing Wang \cite{Wang1} and their papers \cite{Bonnefont2} and \cite{Wang2}. Besides  the  Hopf fibrations, there are of course many other situations where sub-Riemannian heat kernels may computed more or less explicitely. We mention in particular the reference  \cite{ABGR} which deals with the case of unimodular Lie groups.

\subsection{Heisenberg group}

One of the simplest non trivial Riemannian submersions with totally geodesic fibers and bracket generating horizontal distribution is associated to the Heisenberg group.  The Heisenberg group is the set
\[
\mathbb{H}^{2n+1}=\left\{ (x,y,z), x \in \mathbb{R}^n,  y \in \mathbb{R}^n, z\in\mathbb{R} \right\}
\]
endowed with the group law
\[
(x_1,y_1,z_1) \star (x_2,y_2,z_2)=(x_1+x_2,y_1+y_2,z_1+z_2+\langle x_1,y_2 \rangle_{\R^n} -\langle x_2,y_1 \rangle_{\R^n}).
\]
The vector fields
\[
X_i=\frac{\partial}{\partial x_i} -y_i \frac{\partial}{\partial z} 
\]
\[
Y_i=\frac{\partial}{\partial y_i} +x_i \frac{\partial}{\partial z} 
\]
and
\[
Z=\frac{\partial}{\partial z}
\]
form an orthonormal frame of left invariant vector fields for the left invariant metric on $\mathbb{H}^{2n+1}$. Note that the following commutations hold
\[
[X_i,Y_j]=2\delta_{ij} Z, \quad [X_i,Z]=[Y_i,Z]=0.
\]
The map
\begin{align*}
\pi :
\begin{array}{lll}
  \mathbb{H}^{2n+1} &\to& \mathbb{R}^{2n} \\
 (x,y,z) & \to & (x,y)
 \end{array}
\end{align*}
is then a Riemannian submersion with totally geodesic fibers. The horizontal Laplacian is  the left invariant operator
\begin{align*}
\Delta_{\mathcal{H}} &=\sum_{i=1}^n (X_i^2+Y_i^2) \\
 &=\sum_{i=1}^n \frac{\partial^2}{\partial x^2_i} +\frac{\partial^2}{\partial y^2_i} + 2\sum_{i=1}^n \left( x_i \frac{\partial}{\partial y_i}-y_i \frac{\partial}{\partial x_i}\right) \frac{\partial}{\partial z}+ (\| x\|^2+\| y \|^2) \frac{\partial^2}{\partial z^2} 
\end{align*}
and the vertical Laplacian is the left invariant operator
\[
\Delta_\mathcal{V}=\frac{\partial^2}{\partial z^2} .
\]
The horizontal distribution
\[
\mathcal{H}=\mathbf{span} \{X_1, \cdots,X_n,Y_1,\cdots, Y_n\}
\]
is bracket generating at every point, so $\Delta_{\mathcal{H}} $ is a subelliptic operator. The operator $\Delta_{\mathcal{H}} $ is invariant by the action of the  orthogonal group  of $\mathbb{R}^{2n}$ on the variables $(x,y)$. Introducing the variable $r^2=\| x\|^2 +\|y\|^2$, we see then that the radial part of $\Delta_{\mathcal{H}} $ is given by
\[
\rL=\frac{\partial^2}{\partial r^2}+\frac{2n-1}{r} \frac{\partial}{\partial r}  +r^2 \frac{\partial^2}{\partial z^2}.
\]
This means that if $f: \mathbb{R}_{\ge 0} \times \R \to \mathbb R$ is a smooth map and $\rho$ is the submersion $(x,y,z)\to (\sqrt{\| x\|^2+\|y\|^2},z)$ then
\[
\Delta_{\mathcal{H}} (f \circ \rho)=(\rL f) \circ \rho.
\]
From this invariance property in order to study the heat kernel and fundamental solution of $\Dh$ at $0$ it suffices to study the heat kernel and the fundamental solution of $\rL$ at $0$.

\

We denote by $\rp_t(r,z)$ the heat kernel at 0 of $\rL$. It was first computed explicitly by Gaveau \cite{Gaveau},  building on previous works by Paul L\'evy. 

\begin{proposition}
For $r \ge 0$ and $z \in \mathbb{R}$,
\[
\rp_t (r,z)=\frac{1}{(2\pi)^{n+1}} \int_\R e^{i \lambda z} \left( \frac{\lambda}{\sinh (2\lambda t)} \right)^n e^{-\frac{\lambda r^2}{ 2} \coth (2\lambda t) } d\lambda
\]
\end{proposition}

\begin{proof}
Since $ \frac{\partial}{\partial z} $ commutes with $\rL$, the idea is to use a Fourier transform in $z$. We see then that
\[
\rp_t (r,z)=\frac{1}{2\pi} \int_\R e^{i \lambda z} \Phi_t (r,\lambda) d\lambda,
\]
where $ \Phi_t (r,z,\lambda)$ is the fundamental solution at 0 of the parabolic partial differential equation
\[
\frac{\partial \Phi}{ \partial t}=\frac{\partial^2 \Phi }{\partial r^2}+\frac{2n-1}{r} \frac{\partial \Phi }{\partial r}  -\lambda^2 r^2 \Phi .
\]
We thus want to compute the semigroup generated by the Sch\"rodinger operator
\[
\mathcal{L}_\lambda=\frac{\partial^2  }{\partial r^2}+\frac{2n-1}{r} \frac{\partial  }{\partial r}  -\lambda^2 r^2.
\]
The trick is now to observe that for every $f$,
\[
\mathcal{L}_\lambda \left( e^{\frac{\lambda r^2}{2}} f \right)= e^{\frac{\lambda r^2}{2}} \left( 2n\lambda +\mathcal{G}_\lambda \right)f,
\]
where
\[
\mathcal{G}_\lambda=\frac{\partial^2   }{\partial r^2}+\left( 2 \lambda r+\frac{2n-1}{r} \right)\frac{\partial   }{\partial r}.
\]
The operator $\mathcal{G}_\lambda$ turns out to be the radial part of the Ornstein-Uhlenbeck operator $\Delta_{\mathbb{R}^{2n}} +2 \lambda \langle x , \nabla_{\mathbb{R}^{2n}} \rangle$ whose heat kernel at 0 is a  Gaussian density with mean 0 and variance $\frac{1}{2\lambda}(e^{4\lambda t}-1)$. This means that the heat kernel at 0 of $\mathcal{G}_\lambda$ is given by
\[
q_t (r)=\frac{1}{(2\pi)^{n}} \left( \frac{2\lambda}{e^{4\lambda t}-1} \right)^n e^{-\frac{\lambda r^2}{ e^{4\lambda t}-1}}.
\]
We conclude
\[
 \Phi_t (r,z,\lambda)=\frac{e^{2n\lambda t}}{(2\pi)^n} \left( \frac{2\lambda}{e^{4\lambda t}-1} \right)^n e^{-\frac{\lambda r^2}{2}}  e^{-\frac{\lambda r^2}{ e^{4\lambda t}-1}}
 \]
\end{proof}
As a straightforward  corollary, we deduce the heat kernel at 0 of $\Dh$.

\begin{corollary}
The heat kernel at 0 of $\Dh$ is 
\[
p_t(x,y,z)=\frac{1}{(2\pi)^{n+1}} \int_\R e^{i \lambda z} \left( \frac{\lambda}{\sinh (2\lambda t)} \right)^n e^{-\frac{\lambda (\| x\|^2+\|y\|^2)}{ 2} \coth (2\lambda t) } d\lambda
\]
\end{corollary}

Though it does not seem very explicit, this representation of the heat kernel has many applications and can be used to get very sharp estimates and small-time asymptotics (see \cite{BGG} and \cite{Li1,Li2}).

\subsection{The Hopf fibration}

The second simplest and geometrically relevant example is given by the celebrated Hopf fibration. The horizontal heat kernel was first computed in \cite{BW1} that we follow but simplify since, here, the CR structure of the sphere is not relevant for us.

\

Let us consider the odd dimensional unit sphere 
\[
\bS^{2n+1}=\lbrace z=(z_1,\cdots,z_{n+1})\in \mathbb{C}^{n+1}, \| z \| =1\rbrace.
\]
There is an isometric group action of $\mathbb{S}^1=\mathbf{U}(1)$ on $\bS^{2n+1}$ which is  defined by $$(z_1,\cdots, z_n) \rightarrow (e^{i\theta} z_1,\cdots, e^{i\theta} z_n). $$ The generator of this action shall be denoted by $T$. We thus have for every $f \in C^\infty(\bS^{2n+1})$
\[
Tf(z)=\frac{d}{d\theta}f(e^{i\theta}z)\mid_{\theta=0},
\]
so that
\[
T=i\sum_{j=1}^{n+1}\left(z_j\frac{\partial}{\partial z_j}-\overline{z_j}\frac{\partial}{\partial \overline{z_j}}\right).
\]
The quotient space $\bS^{2n+1} / \mathbf{U}(1)$ is the projective complex space $\mathbb{CP}^n$ and the projection map $\pi :  \bS^{2n+1} \to \mathbb{CP}^n$ is a Riemannian submersion with totally geodesic fibers isometric to $\mathbf{U}(1)$. The fibration
\[
\mathbf{U}(1) \to \bS^{2n+1} \to \mathbb{CP}^n
\]
 is called the Hopf fibration.

\

To study the geometry of the Hopf fibration, in particular the horizontal Laplacian  $\Dh$, it is convenient  to introduce a set of coordinates that reflects the action of the isometry group of $\mathbb{CP}^n$ on $\mathbb{S}^{2n+1}$.
Let $(w_1,\cdots, w_n,\theta)$ be  the local inhomogeneous coordinates for $\mathbb{CP}^n$ given by $w_j=z_j/z_{n+1}$, and $\theta$ be the local fiber coordinate. i.e., $(w_1, \cdots, w_n)$ parametrizes the complex lines passing through the north pole\footnote{We will call north pole the point with complex coordinates $z_1=0,\cdots, z_{n+1}=1$. }, while $\theta$ determines a point on the line that is of unit distance from the north pole. More explicitly, these coordinates are given by the map
\begin{align}\label{cylinder}
(w,\theta)\longrightarrow \left(w e^{i\theta}\cos r ,e^{i\theta}\cos r \right),
\end{align}
where $r=\arctan \sqrt{\sum_{j=1}^{n}|w_j|^2} \in [0,\pi /2)$, $\theta \in \R/2\pi\mathbb{Z}$, and $w \in \mathbb{CP}^n$. In these coordinates, it is clear that $T=\frac{\partial}{\partial \theta}$ and that the vertical Laplacian is
\[
\Dv=\frac{\partial^2 }{\partial \theta^2}.
\]
Our goal is now to compute the  horizontal Laplacian $\Dh$.  This operator is invariant  by the action on the variables $(w_1,\cdots,w_n)$ of the group of isometries of $\mathbb{CP}^n$ that fix the north pole of $\mathbb{S}^{2n+1}$ (this group is $\mathbf{SU}(n)$). Therefore the heat kernel at the north pole only depends on the variables $(r,\theta)$ and can be computed through the heat of kernel of the radial part $\rL$ of $\Dh$.

\begin{proposition}
Consider the submersion
\begin{align*}
\rho:
\begin{array}{lll}
\mathbb{S}^{2n+1} & \to & [0,\pi/2) \times  \R/2\pi\mathbb{Z} \\
(w,\theta)  & \to & (r,\theta)
\end{array}
\end{align*}
where we recall that $r=\arctan \sqrt{\sum_{j=1}^{n}| w_j|^2}$.  Then for every smooth map $f: [0,\pi/2) \times  \R/2\pi\mathbb{Z} \to \R$,
\[
\Delta_{\mathcal{H}} (f \circ \rho)=(\rL f) \circ \rho,
\]
where
\begin{align*}
\rL=\frac{\partial^2}{\partial r^2}+((2n-1)\cot r-\tan r)\frac{\partial}{\partial r}+\tan^2r\frac{\partial^2}{\partial \theta^2}.
\end{align*}
\end{proposition}

\begin{proof}
The easiest route is to compute first the radial part of the Laplace-Beltrami operator $\Delta$ and then to use the formula
\[
\Dh=\Delta-\Dv=\Delta-\frac{\partial}{\partial \theta^2}.
\]
In our parametrization of $\mathbb{S}^{2n+1}$ we have,
\[
z_{n+1}=e^{i\theta}\cos r. 
\]
Therefore if $\delta_1$ denotes the Riemannian distance based at the north pole, we have $\cos \delta_1 =\cos r \cos \theta$ and if $\delta_2$ denotes the Riemannian distance based at the point with real coordinates $(0,\cdots,0,1)$ then we have $\cos \delta_2=\cos r  \sin \theta$. The formula for the Laplace-Beltrami operator acting on functions depending on the Riemannian distance based at a point is well-known and we deduce from it that $\Delta$ acts on functions depending only on $\delta_1,\delta_2$ as
\[
\frac{\partial^2}{\partial \delta_1^2} + 2n \cot \delta_1 \frac{\partial }{\partial \delta_1} +\frac{\partial^2}{\partial \delta_2^2} + 2n \cot \delta_2  \frac{\partial }{\partial \delta_2}
\]
In the variables $(r,\theta)$ this last operator writes
\[
\frac{\partial^2}{\partial r^2}+((2n-1)\cot r-\tan r)\frac{\partial}{\partial r}+\frac{1}{\cos^2 r}\frac{\partial^2}{\partial \theta^2}.
\]
Thus, we conclude
\[
\rL=\frac{\partial^2}{\partial r^2}+((2n-1)\cot r-\tan r)\frac{\partial}{\partial r}+\frac{1}{\cos^2 r}\frac{\partial^2}{\partial \theta^2} -\frac{\partial}{\partial \theta^2}
\]
\end{proof}

We can observe that  $\rL$ is symmetric with respect to the measure
\[
d\overline{\mu}=\frac{2\pi^n}{\Gamma(n)}(\sin r)^{2n-1}\cos r drd\theta,
\]
where the normalization is chosen in such a way that
\[
\int_{-\pi}^{\pi}\int_0^{\frac{\pi}{2}}d\overline{\mu}=\mu(\bS^{2n+1})=\frac{2\pi^{n+1}}{\Gamma (n+1)}.
\]

As mentioned above, the heat kernel at the north pole of $\Dh$ only depends on $(r, \theta)$, that is $p\left(w e^{i\theta}\cos r ,e^{i\theta}\cos r \right)=\rp_t(r, \theta)$, where $\rp_t$ is the heat kernel at 0 of $\rL$. 

\begin{prop}\label{heatHopf}
For $t>0$, $r\in[0,\frac{\pi}{2})$, $ \theta\in[-\pi,\pi]$:
\[
\rp_t(r, \theta)=\frac{\Gamma(n)}{2\pi^{n+1}}\sum_{k=-\infty}^{+\infty}\sum_{m=0}^{+\infty} (2m+|k|+n){m+|k|+n-1\choose n-1}e^{-\lambda_{m,k}t+ik \theta}(\cos r)^{|k|}P_m^{n-1,|k|}(\cos 2r),
\]
where $\lambda_{m,k}=4m(m+|k|+n)+2|k|n$ and
\[
P_m^{n-1,|k|}(x)=\frac{(-1)^m}{2^m m!(1-x)^{n-1}(1+x)^{|k|}}\frac{d^m}{dx^m}((1-x)^{n-1+m}(1+x)^{|k|+m})
\]
is a Jacobi polynomial.
\end{prop}

\begin{proof}
Similarly to the Heisenberg group case, we observe that $\Dh$ commutes with $\frac{\partial}{\partial \theta}$, so the idea is to expand $p_t(r, \theta)$ as a Fourier series in $\theta$. We can write
\[
\rp_t(r, \theta)=\frac{1}{2\pi} \sum_{k=-\infty}^{+\infty} e^{ik\theta}\phi_k(t,r),
\]
where $\phi_k$ is the fundamental solution at 0 of the parabolic equation
\[
\frac{\partial\phi_k}{\partial t}=\frac{\partial^2\phi_k}{\partial r^2}+((2n-1)\cot r-\tan r)\frac{\partial\phi_k}{\partial r}-k^2\tan^2 r\phi_k.
\]
By writing $\phi_k(t,r)$ in the form
\[
\phi_k(t,r)=e^{-2n|k|t}(\cos r)^{|k|}g_k(t, \cos 2r), 
\]
we get 
\begin{equation*}
\frac{\partial g_k}{\partial t}=4\mathcal{L}_k(g_k),
\end{equation*}
where
\[
\mathcal{L}_k=(1-x^2)\frac{\partial^2}{\partial x^2}+[(|k|+1-n)-(|k|+1+n)x]\frac{\partial}{\partial x}.
\]
The eigenvectors of $\mathcal{L}_k$ solve the Jacobi differential equation, and are thus given by the Jacobi polynomials
\[
P_m^{n-1,|k|}(x)=\frac{(-1)^m}{2^m m!(1-x)^{n-1}(1+x)^{|k|}}\frac{d^m}{dx^m}((1-x)^{n-1+m}(1+x)^{|k|+m}),
\]
which satisfy
\[
\mathcal{L}_k(P_m^{n-1,|k|})(x)=-m(m+n+|k|)P_m^{n-1,|k|}(x).
\]

By using the fact that the family $(P_m^{n-1,|k|}(x)(1+x)^{|k|/2})_{m\geq0}$ is an orthogonal basis of $L^2([-1,1],(1-x)^{n-1}dx)$, such that
\[
\int_{-1}^1 P_m^{n-1,|k|}(x)P_l^{n-1,|k|}(x)(1-x)^{n-1}(1+x)^{|k|}dx=\frac{2^{n+|k|}}{2m+|k|+n}\frac{\Gamma(m+n)\Gamma(m+|k|+1)}{\Gamma(m+1)\Gamma(m+n+|k|)}\delta_{ml},
\]
we easily compute the fundamental solution of the operator $\frac{\partial }{\partial t}- 4\mathcal{L}_k$ and thus $\rp_t$.
\end{proof}

\

Note that as a by-product of the previous result we obtain that the $L^2$ spectrum of $-\Dh$ is given by
\begin{align}\label{spectreS1}
\mathbf{Sp} (-\Dh) =\left\{ 4m(m+k+n)+2kn, k \in \mathbb{N}, m \in \mathbb{N} \right\}.
\end{align}

\

We can give  another representation of the heat kernel $\rp_t(r,\theta)$ which is easier to handle analytically. The key idea is to observe that since $\Delta$ and $\frac{\partial}{\partial \theta}$ commute, we formally  have
\begin{align}\label{commutation}
e^{t\Dh}=e^{-t\frac{\partial^2}{\partial\theta^2}}e^{t\Delta}.
\end{align}

This gives a way to express the horizontal heat kernel in terms of the Riemannian one.
Let us recall that the Riemannian heat kernel on the sphere $\bS^{2n+1}$ is given by
\begin{equation}\label{eq6}
q_t(\cos\delta)=\frac{\Gamma(n)}{2\pi^{n+1}}\sum_{m=0}^{+\infty}(m+n)e^{-m(m+2n)t}C_m^n(\cos \delta),
\end{equation}
where, $\delta$ is the Riemannian distance based at the north pole and
\[
C_m^n(x)=\frac{(-1)^m}{2^m}\frac{\Gamma(m+2n)\Gamma(n+1/2)}{\Gamma(2n)\Gamma(m+1)\Gamma(n+m+1/2)}\frac{1}{(1-x^2)^{n-1/2}}\frac{d^m}{dx^m}(1-x^2)^{n+m-1/2},
\]
is a Gegenbauer polynomial.  Another expression of $q_t (\cos \delta)$  is 
\begin{equation}\label{heat_kernel_odd}
q_t (\cos \delta)= e^{n^2t} \left( -\frac{1}{2\pi \sin \delta} \frac{\partial}{\partial \delta} \right)^n V
\end{equation}
where $V(t,\delta)=\frac{1}{\sqrt{4\pi t}} \sum_{k \in \mathbb{Z}} e^{-\frac{(\delta-2k\pi)^2}{4t} }$ is a theta function.

Using the commutation (\ref{commutation}) and the formula $\cos \delta =\cos r \cos \theta$, we then infer the following proposition which  is easy to prove (see \cite{BW1} for the details).
\begin{prop}\label{prop1}
For $t>0$, $r\in[0,\pi/2)$, $ \theta\in[-\pi,\pi]$, 
\begin{equation}\label{eq8}
\rp_t(r, \theta)=\frac{1}{\sqrt{4\pi t}}\int_{-\infty}^{+\infty}e^{-\frac{(y+i \theta)^2}{4t} }q_t(\cos r\cosh y)dy.
\end{equation}
\end{prop}

Applications of this formula are given in \cite{BW1}. We can, in particular, deduce from it small asymptotics of the kernel when $t \to 0$. Interestingly, these small-time asymptotics allow to compute explicitly the sub-Riemannian distance. For a study of the distance and related geodesics, we refer to \cite{CMV} and \cite{MM}.

\subsection{The quaternionic Hopf fibration} 

We study now a second example of Riemannian submersion with totally geodesic fibers and compact base: the quaternionic Hopf fibrration. Up to exotic examples, the Hopf fibration and the quaternionic Hopf fibration are the only Riemannian submersions of the sphere with totally geodesic fibers (see \cite{escobales}). The computation of the horizontal heat kernel was first done in \cite{BW2}.

\

Let
\[
\mathbb{H} =\{ q=t+x I +y J +z K, (t,x,y,z) \in \mathbb{R}^4 \},
\]
be the field of quaternions,  where $I,J,K$ are the Pauli matrices:
\[
I=\left(
\begin{array}{ll}
i & 0 \\
0&-i 
\end{array}
\right), \quad 
J= \left(
\begin{array}{ll}
0 & 1 \\
-1 &0 
\end{array}
\right), \quad 
K= \left(
\begin{array}{ll}
0 & i \\
i &0 
\end{array}
\right).
\]
The quaternionic norm is given by
\[
\| q \|^2 =t^2 +x^2+y^2+z^2.
\]

Consider now the quaternionic unit sphere which is given by
\[
\bS^{4n+3}=\lbrace q=(q_1,\cdots,q_{n+1})\in \mathbb{H}^{n+1}, \sum_{i=1}^{n+1} \| q_i \|^2 =1\rbrace.
\]

 There is an isometric group action of the Lie group $\mathbf{SU}(2)$ on $\bS^{4n+3}$ which is given by, 
 $$ g \cdot (q_1,\cdots, q_{n+1}) = (g q_1,\cdots, g q_{n+1}). $$
 
 The three generators of this action are given by
 \[
 \frac{d}{d\theta}f(e^{I\theta}q)\mid_{\theta=0}=\sum_{i=1}^{n+1}\left(-x_i\frac{\partial f}{\partial t_i}+t_i\frac{\partial f}{\partial x_i}-z_i\frac{\partial f}{\partial y_i}+y_i\frac{\partial f}{\partial z_i}\right),
 \]
  \[
 \frac{d}{d\theta}f(e^{J\theta}q)\mid_{\theta=0}=\sum_{i=1}^{n+1}\left(-y_i\frac{\partial f}{\partial t_i}+z_i\frac{\partial f}{\partial x_i}+t_i\frac{\partial f}{\partial y_i}-x_i\frac{\partial f}{\partial z_i}\right).
 \]
 and
 \[
 \frac{d}{d\theta}f(e^{K\theta}q)\mid_{\theta=0}=\sum_{i=1}^{n+1}\left(-z_i\frac{\partial f}{\partial t_i}-y_i\frac{\partial f}{\partial x_i}+x_i\frac{\partial f}{\partial y_i}+t_i\frac{\partial f}{\partial z_i}\right).
 \]

The quotient space $\bS^{4n+3} / \mathbf{SU}(2)$ is the projective quaternionic space $\mathbb{HP}^n$ and the projection map $\pi :  \bS^{4n+3} \to \mathbb{HP}^n$ is a Riemannian submersion with totally geodesic fibers isometric to $\mathbf{SU}(2)$. The fibration
\[
\mathbf{SU}(2) \to \bS^{4n+3} \to \mathbb{HP}^n
\]
is called the quaternionic Hopf fibration.

\

As for the classical Hopf fibration, the first task is to introduce a convenient set of coordinates. Let  $(w_1,\cdots,w_n)$ be the local inhomogeneous coordinates for $\mathbb{HP}^n$ given by $w_j=q_{n+1}^{-1} q_j$ and $\theta_1,\theta_2,\theta_3$ be the local exponential coordinates on the $\mathbf{SU}(2)$ fiber. We can locally parametrize $\bS^{4n+3}$ by the coordinates

\begin{align}\label{cylinder}
(w,\theta_1,\theta_2,\theta_3 )\longrightarrow \left( (\cos r)  e^{ I\theta_1 +J\theta_2 +K\theta_3} w  ,(\cos r) e^{ I\theta_1 +J\theta_2 +K\theta_3} \right),
\end{align}
where $r=\arctan \sqrt{\sum_{j=1}^{n}|w_j|^2}$.

\

The horizontal Laplacian $\Dh$ is invariant by the action on the variable $w$ of the group of  isometries of $\mathbb{HP}^n$ that fix the north pole of  $\bS^{4n+3}$ and by the action on the variables $\theta_1,\theta_2,\theta_3$ of the group of isometries of $\mathbf{SU}(2)$ that fix the identity. Thus the heat kernel of $\Dh$ only depends on the variables $r=\arctan \sqrt{\sum_{j=1}^{n}|w_j|^2}$ and $\eta= \sqrt{ \theta_1^2+\theta_2^2+\theta_3^2}$. Observe that $\eta$ is the distance based at the identity in  $\mathbf{SU}(2)$ because
\[
e^{ I\theta_1 +J\theta_2 +K\theta_3}=\cos \eta +\frac{\sin \eta}{\eta} \left(  I\theta_1 +J\theta_2 +K\theta_3 \right).
\]

\begin{proposition} 
Let us denote by $\rho$ the submersion from $\bS^{4n+3} $ to $ [0,\pi/2) \times [0,\pi ) $ such that 
\[
\rho  \left( (\cos r)  e^{ I\theta_1 +J\theta_2 +K\theta_3} w  ,(\cos r) e^{ I\theta_1 +J\theta_2 +K\theta_3} \right)=\left(r, \eta \right),
\]
where  $r=\arctan \sqrt{\sum_{j=1}^{n}|w_j|^2}$ and $\eta= \sqrt{ \theta_1^2+\theta_2^2+\theta_3^2}$. Then for every smooth function $f:  [0,\pi/2) \times [0,\pi ) \to \R$,
\[
\Dh(f \circ \rho)=(\rL f) \circ \rho, \quad \Delta_\mathcal{V} ( f \circ \rho) = (\overline{\Delta}_\mathcal{V}  f) \circ \rho
\]
where
\begin{equation*}
\rL=\frac{\partial^2}{\partial r^2}+((4n-1)\cot r-3\tan r)\frac{\partial}{\partial r}+\tan^2r \left(\frac{\partial^2}{\partial \eta^2}+2\cot \eta\frac{\partial}{\partial \eta}\right).
\end{equation*}
and 
\begin{align*}
\overline{\Delta}_\mathcal{V}=\frac{\partial^2}{\partial \eta^2}+2\cot \eta\frac{\partial}{\partial \eta}.
\end{align*}
\end{proposition}

\begin{proof}
The formula for $\overline{\Delta}_\mathcal{V}$ is clear because $\mathbf{SU}(2)$ is isometric to the sphere $\bS^3$. For the horizontal Laplacian, the proof follows the same lines as in the case of the classical Hopf fibration.  Let $\delta_1$ be the distance based at the point $(0,1) \in \mathbb{H}^n \times \mathbb{H}$, $\delta_2$ be the distance based at the point $(0,I) \in \mathbb{H}^n \times \mathbb{H}$, $\delta_3$ be the distance based at the point $(0,J) \in \mathbb{H}^n \times \mathbb{H}$, and $\delta_4$ be the distance based at the point $(0,K) \in \mathbb{H}^n \times \mathbb{H}$.

The Laplace-Beltrami operator $\Delta$ acts on functions depending only on $\delta_1,\delta_2,\delta_3,\delta_4$ as
\[
\sum_{i=1}^4 \left(  \frac{\partial^2}{\partial \delta_i^2} + (4n+2) \cot \delta_i \frac{\partial }{\partial \delta_i} \right)
\]
Observing now that
\begin{align*}
\begin{cases}
\cos r =\sqrt{\cos^2 \delta_1+\cos^2 \delta_2 +\cos^2 \delta_3+\cos^2 \delta_4}\\
\tan \eta=\frac{\sqrt{\cos^2 \delta_2 +\cos^2 \delta_3+\cos^2 \delta_4}}{\cos \delta_1}
\end{cases}
\end{align*}
finishes the proof after a simple, but tedious, change of variables.
\end{proof}

As a consequence of the previous result, we can check that the Riemannian measure of $\bS^{4n+3} $  in the coordinates $(r,\eta)$,  which is the symmetric measure for  $\rL$ is given by
\[
d\overline{\mu}=\frac{8\pi^{2n+1}}{\Gamma(2n)}(\sin r)^{4n-1}(\cos r)^3(\sin\eta)^2drd\eta.
\]

As before, we denote by $\rp_t$ the heat kernel at 0 of $\rL$.
\begin{proposition}\label{heatquaternion}
For $t>0$, $r\in[0,\frac{\pi}{2})$, $ \eta\in[0,\pi]$, 
\begin{equation}\label{pt}
\rp_t(r,\eta)=\sum_{m=0}^{+\infty}\sum_{k=0}^{\infty}\alpha_{k,m}e^{-\lambda_{k,m}t}\frac{\sin (m+1)\eta}{\sin \eta}(\cos r)^mP_k^{2n-1,m+1}(\cos 2r)
\end{equation}
where 
\[
\alpha_{k,m}=\frac{\Gamma(2n)}{2\pi^{2n+2}}(2k+m+2n+1)(m+1){k+m+2n\choose 2n-1}, 
\]
\[
\lambda_{k,m}=4\left[ k(k+2n+m+1)+nm\right],
\]
 and 
\[
P_k^{2n-1,m+1}(x)=\frac{(-1)^k}{2^kk!(1-x)^{2n-1}(1+x)^{m+1}}\frac{d^k}{dx^k}\left((1-x)^{2n-1+k}(1+x)^{m+1+k} \right).
\] 
is a Jacobi polynomial.
\end{proposition} 
\begin{proof}
The idea is to   expand the subelliptic kernel in spherical harmonics as follows, 
\[
p_t(r,\eta)=\sum_{m=0}^{+\infty}\frac{\sin (m+1)\eta}{\sin \eta}\phi_m(t,r)
\]
where $\frac{\sin (m+1)\eta}{\sin \eta}$ is the eigenfunction of $\tilde{\Delta}_{SU(2)}=\frac{\partial^2}{\partial \eta^2}+2\cot \eta\frac{\partial}{\partial \eta}$ which is associated to the eigenvalue $-m(m+2)$. To determine $\phi_m$, we use  $\frac{\partial p_t}{\partial t}=\tilde{L}p_t$ and find that
\[
\frac{\partial\phi_m}{\partial t}=\frac{\partial^2\phi_m}{\partial r^2}+\left((4n-1)\cot r-3\tan r \right)\frac{\partial\phi_m}{\partial r}-m(m+2)\tan^2r\phi_m.
\]
Let $\phi_m(t,r)=e^{-4nmt}(\cos r)^m\varphi_m(t,r)$, then $\varphi_m(t,r)$ satisfies the equation
\[
\frac{\partial\varphi _m}{\partial t}=\frac{\partial^2\varphi_m}{\partial r^2}+[(4n-1)\cot r-(2m+3)\tan r]\frac{\partial\varphi_m}{\partial r}.
\]
We now change the variable and denote by $\varphi_m(t,r)=g_m(t,\cos 2r)$, then we have that $g_m(t,x)$ satisfies the equation
\[
\frac{\partial g_m}{\partial t}=4(1-x^2)\frac{\partial^2 g_m}{\partial x^2}+4[(m+2-2n)-(2n+m+2)x]\frac{\partial g_m}{\partial x}.
\]
We denote $\Psi_m=(1-x^2)\frac{\partial^2 }{\partial x^2}+[(m+2-2n)-(2n+m+2)x]\frac{\partial}{\partial x}$, and find that
\[
\frac{\partial g_m}{\partial t}=4\Psi_m(g_m).
\]
The equation $$\Psi_m(g_m)+k(k+2n+m+1)g_m=0$$ is a Jacobi differential equation for all $k\geq 0$.  We denote the eigenvector of $\Psi_m$ corresponding to the eigenvalue $-k(k+2n+m+1)$ by $P_k^{2n-1,m+1}(x)$, then it is known that 
\[
P_k^{2n-1,m+1}(x)=\frac{(-1)^k}{2^kk!(1-x)^{2n-1}(1+x)^{m+1}}\frac{d^k}{dx^k}\left((1-x)^{2n-1+k}(1+x)^{m+1+k} \right).
\] 
At the end we can therefore write the spectral decomposition as
\[
p_t(r,\eta)=\sum_{m=0}^{+\infty}\sum_{k=0}^{\infty}\alpha_{k,m}e^{-4[k(k+2n+m+1)+nm]t}\frac{\sin (m+1)\eta}{\sin \eta}(\cos r)^mP_k^{2n-1,m+1}(\cos 2r)
\]
where $\alpha_{k,m}$ are  determined by considering the initial condition.

Note that $(P_k^{2n-1,m+1}(x)(1+x)^{(m+1)/2})_{k\geq0}$ is an orthogonal basis of the Hilbert space $L^2([-1,1],(1-x)^{2n-1}dx)$, more precisely
\begin{align*}
 & \int_{-1}^1 P_k^{2n-1,m+1}(x)P_l^{2n-1,m+1}(x)(1-x)^{2n-1}(1+x)^{m+1}dx \\
 =& \frac{2^{2n+m+1}}{2k+m+2n+1}\frac{\Gamma(k+2n)\Gamma(k+m+2)}{\Gamma(k+1)\Gamma(k+2n+m+1)}\delta_{kl}.
\end{align*}
For a smooth function $f(r, \theta)$, we can write
\[
f(r, \eta)=\sum_{m=0}^{+\infty}\sum_{k=0}^{+\infty} b_{k,m}\frac{\sin (m+1)\eta}{\sin \eta}P_k^{2n-1,m+1}(\cos 2r)\cdot(\cos r)^{m}
\]
where the $ b_{k,m}$'s are constants. We obtain then
\[
f(0,0)=\sum_{m=0}^{+\infty}\sum_{k=0}^{+\infty} b_{k,m}(m+1)P_k^{2n-1,m+1}(1).
\]
and we observe that $P_k^{2n-1,m+1}(1)={2n-1+k\choose k}$.
The  measure $d\mu$ is given in cylindric coordinates by
\[
d\mu_r=\frac{8\pi^{2n+1}}{\Gamma(2n)}(\sin r)^{4n-1}(\cos r)^3(\sin\eta)^2drd\eta
\]
Moreover, since
\begin{eqnarray*}
& &\int_{0}^\pi\int_0^\frac{\pi}{2} p_t(r, \eta){f(-r, -\eta)}d\mu_r \\
&=&\frac{4\pi^{2n+2}}{\Gamma(2n)}\sum_{m=0}^{+\infty}\sum_{k=0}^{+\infty}\alpha_{k,m}b_{k,m}e^{-\lambda_{k,m}t}
\left(\int_0^{\frac{\pi}{2}}(\cos r)^{2m+3}|P_k^{2n-1,m+1}|^2(\sin r)^{4n-1}dr\right)\\
&=&\frac{2\pi^{2n+2}}{\Gamma(2n)} \sum_{m=0}^{+\infty}\sum_{k=0}^{+\infty}\frac{\alpha_{k,m}b_{k,m}e^{-\lambda_{m,k}t}}{2k+m+2n+1}\frac{\Gamma(k+2n)\Gamma(k+m+2)}{\Gamma(k+1)\Gamma(k+2n+m+1)}
\end{eqnarray*}
where $\lambda_{k,m}=4k(k+2n+m+1)+nm$, we obtain that
\[
\lim_{t\rightarrow 0}\int_{0}^\pi\int_0^\frac{\pi}{2} p_tfd\mu_r= f(0,0)
\]
as soon as $\alpha_{k,m}=\frac{\Gamma(2n)}{2\pi^{2n+2}}(2k+m+2n+1)(m+1){k+m+2n\choose 2n-1}$.\end{proof}
As a byproduct of the spectral expansion of $\rp_t$ we obtain the spectrum of $-\Dh$,
\begin{align}\label{spectreSU}
\mathbf{Sp}(-\Dh)=\{4\left[k(k+2n+m+1)+nm\right], k\ge0, m \ge 0\}.
\end{align}

Comparing this expansion with the result we obtained in  Proposition \ref{heatHopf}, we obtain a very nice formula relating $\rp_t$ to the horizontal kernel of the usual Hopf fibration.

\begin{proposition}\label{relationHopf}
Let $\hat{p}_t(r, \theta)$ be the radial horizontal kernel of the usual  Hopf fibration   $\bS^{4n+1} \to  \mathbb{CP}^{2n}$ , then for $r\in[0,\frac{\pi}{2})$, $ \theta\in[0,\pi]$,
\begin{equation}\label{pt-qt}
\rp_t(r,\theta)=-\frac{e^{4nt}}{2\pi\sin \theta\cos r}\frac{\partial}{\partial \theta} \hat{p}_t(r, \theta).
\end{equation}
\end{proposition}

As in the case of the usual Hopf fibration, we can obtain an alternative representation $\rp_t(r,\theta)$ which we derive from the decomposition $\Delta=\Dh+\Dv$.

\

We denote by $q_t (\cos \delta)$ the heat kernel at 0  of the  operator $ \frac{\partial^2}{\partial \delta^2} + (4n+2) \cot \delta \frac{\partial }{\partial \delta}$. We recall that
 \begin{equation}
q_t{(\cos\delta)}=\frac{\Gamma(2n+1)}{2\pi^{2n+2}}\sum_{m=0}^{+\infty}(m+2n+1)e^{-m(m+4n+2)t}C_m^{2n+1}(\cos \delta),
\end{equation}
where $\delta$ is the Riemannian distance based at the north pole and
\[
C_m^{2n+1}(x)=\frac{(-1)^m}{2^m}\frac{\Gamma(m+4n+2)\Gamma(2n+3/2)}{\Gamma(4n+2)\Gamma(m+1)\Gamma(2n+m+3/2)}\frac{1}{(1-x^2)^{2n+1/2}}\frac{d^m}{dx^m}(1-x^2)^{2n+m+1/2}
\]
is a Gegenbauer polynomial.  

\

 If we denote $\overline{\Delta}_{\mathbf{SL}(2)}=\frac{\partial^2}{\partial \eta ^2}+2 \coth \eta \frac{\partial}{\partial \eta} $, then from the fact that $\overline{\Delta}_\mathcal{V}=\frac{\partial^2}{\partial \eta^2}+2\cot \eta\frac{\partial}{\partial \eta}$,  it is not hard to see that
\begin{equation}\label{pt-qts}
p_t(r,\eta)=(e^{t\overline{\Delta}_{\mathbf{SL}(2)}}f_t)(r,-i\eta),
\end{equation}
where $f_t(\eta)=q_t(\cos r \cos \eta)$. Therefore, an integral representation of $p_t$,  can be obtained from an explicit expression of the heat semigroup $e^{t\tilde{\Delta}_{\mathbf{SL}(2)}}$. 
\begin{lemma}
Let $\overline{\Delta}_{\mathbf{SL}(2)}=\frac{\partial^2}{\partial \eta ^2}+2 \coth \eta \frac{\partial}{\partial \eta} $. For every $f:\mathbb{R}_{\ge 0} \to \mathbb{R}$ in the domain of $\overline{\Delta}_{\mathbf{SL}(2)}$, we have:
\begin{equation}\label{semigroup}
(e^{t\overline{\Delta}_{\mathbf{SL}(2)}} f)(\eta)=\frac{e^{-t}}{\sqrt{\pi t}} \int_0^{+\infty} \frac{ \sinh r \sinh \left(  \frac{\eta r}{2t}\right) }{\sinh \eta} e^{-\frac{r^2+\eta^2}{4t}} f( r ) dr, \quad t \ge 0, \eta \ge 0.
\end{equation}
\end{lemma}

\begin{proof}
Let us take a function  $f$ which is smooth and compactly supported on $\R_{\ge 0}$. We observe that:
\[
\overline{\Delta}_{\mathbf{SL}(2)} f=\frac{1}{h} (\overline{\Delta}_{\R^3}-1)(hf),
\]
where
\[
\overline{\Delta}_{\R^3} =\frac{\partial^2}{\partial \eta ^2}+\frac{2}{  \eta} \frac{\partial}{\partial \eta} , \quad h(\eta)=\frac{\sinh \eta} {\eta}.
\]
As a consequence, we have
\[
(e^{t\overline{\Delta}_{\mathbf{SL}(2)}} f)(\eta)=\frac{e^{-t}}{h(\eta)}e^{t \overline{\Delta}_{\R^3}} (hf) (\eta).
\]
We are thus let with the computation of $e^{t \overline{\Delta}_{\R^3}}$. The operator $\overline{\Delta}_{\R^3}$ is the radial part of the Laplacian $\Delta_{\mathbb{R}^3}$, thus after a routine computation, for $x \in \mathbb{R}^3$,
\[
e^{t\overline{\Delta}_{\R^3}}f (\eta)=\frac{1}{\sqrt{\pi t}} \int_0^{+\infty} \frac{r}{\eta} \sinh \left( \frac{\eta r}{2t}\right) e^{-\frac{r^2+\eta^2}{4t}} f ( r ) dr.
\]
\end{proof}

As a consequence, we get the integral representation of $p_t$.
\begin{proposition}\label{prop-pt}
For $t>0$, $r\in[0,\pi/2)$, $ \eta\in[0,\pi]$, 
\begin{equation*}
p_t(r, \eta)=\frac{e^{-t}}{\sqrt{\pi t}} \int_0^{+\infty} \frac{ \sinh y \sin \left(  \frac{\eta y}{2t}\right) }{\sin \eta} e^{-\frac{y^2-\eta^2}{4t}} q_t( \cos r\cosh y ) dy.
\end{equation*}
\end{proposition}

We refer to \cite{BW2} for applications of this formula to the computation of the small-time asymptotics of the kernel and, as a by-product, of the sub-Riemannian distance.

\section{Transverse Weitzenb\"ock formulas}

In this section we establish a Weitzenb\"ock formula for the horizontal Laplacian of a totally geodesic foliation.  As a consequence, we prove a generalized curvature dimension inequality for the horizontal Laplacian. In a joint program with Nicola Garofalo it has been proved in a very general and abstract framework that the generalized curvature dimension inequality implies several results:
\begin{itemize}
\item Li-Yau type gradient bounds for the heat kernel and associated scale invariant parabolic Harnack inequalities \cite{BG};
\item Upper and lower Gaussian bounds for the heat kernel \cite{BBG, BBGM,BG};
\item Boundedness of the Riesz transform \cite{BG2};
\item Sobolev embeddings and isoperimetric inequalities \cite{BG0,BK};
\item Log-Sobolev and transport inequalities \cite{BB}.
\item Bonnet-Myers type compactness theorem \cite{BG}.
\end{itemize}

We shall not discuss all of these applications here because it would go beyond the scope of these notes, but we will focus on the  Bonnet-Myers compactness result in a later section. In the next section, we will also say some words about the Li-Yau estimates since they are a crucial ingredient in the proof of the Bonnet-Myers type result. 

\

As a second application of the transverse Weitzenb\"ock formula we obtain sharp lower bounds for the first eigenvalue of the horizontal Laplacian.

\subsection{The Bott connection}

Let $\M$ be a smooth, connected  manifold with dimension $n+m$. We assume that $\bM$ is equipped with a Riemannian foliation $\mathcal{F}$ with bundle like metric $g$ and totally geodesic  $m$-dimensional leaves.

\

As usual, the sub-bundle $\mathcal{V}$ formed by vectors tangent to the leaves will be referred  to as the set of \emph{vertical directions} and the sub-bundle $\mathcal{H}$ which is normal to $\mathcal{V}$ will be referred to as the set of \emph{horizontal directions}.   The metric $g$ can be split as
\[
g=g_\mathcal{H} \oplus g_{\mathcal{V}},
\]
We define the canonical variation of $g$ as the one-parameter family of Riemannian metrics:
\[
g_{\varepsilon}=g_\mathcal{H} \oplus  \frac{1}{\varepsilon }g_{\mathcal{V}}, \quad \varepsilon >0.
\]

\

On the Riemannian manifold $(\M,g)$ there is the Levi-Civita connection that we denote by $D$, but this connection is not adapted to the study of follations because the horizontal and  the vertical bundle may not be parallel. More adapted to the geometry of the foliation is the Bott's connection that we now define. It is an easy exercise to check that there exists a unique affine connection $\nabla$ such that:

\begin{itemize}
\item $\nabla$ is metric, that is, $\nabla g =0$;
\item For $X,Y \in \Gamma^\infty(\Ho)$, $\nabla_X Y \in \Gamma^\infty(\Ho)$;
\item For $U,V \in \Gamma^\infty(\V)$, $\nabla_U V \in \Gamma^\infty(\V)$;
\item  For $X,Y \in \Gamma^\infty(\Ho)$, $T(X,Y) \in  \Gamma^\infty(\V)$ and  for $U,V \in \Gamma^\infty(\V)$, $T(U,V) \in  \Gamma^\infty(\Ho)$, where $T$ denotes the torsion tensor of $\nabla$;
\item For $X \in \Gamma^\infty(\Ho), U \in \Gamma^\infty(\V)$, $T(X,U)=0$.
\end{itemize}

In terms of the Levi-Civita connection,  the Bott connection writes

\[
\nabla_X Y =
\begin{cases}
 ( D_X Y)_{\mathcal{H}} , \quad X,Y \in \Gamma^\infty(\mathcal{H}) \\
 [X,Y]_{\mathcal{H}}, \quad X \in \Gamma^\infty(\mathcal{V}), Y \in \Gamma^\infty(\mathcal{H}) \\
 [X,Y]_{\mathcal{V}}, \quad X \in \Gamma^\infty(\mathcal{H}), Y \in \Gamma^\infty(\mathcal{V}) \\
 ( D_X Y)_{\mathcal{V}}, \quad X,Y \in \Gamma^\infty(\mathcal{V})
\end{cases}
\]
where the subscript $\Ho$  (resp. $\mathcal{V}$) denotes the projection on $\mathcal{H}$ (resp. $\mathcal{V}$). Observe that for horizontal vector fields $X,Y$ the torsion $T(X,Y)$ is given by
\[
T(X,Y)=-[X,Y]_\V.
\]
Also observe that for $X,Y \in \Gamma^\infty(\mathcal{V})$ we actually have  $( D_X Y)_{\mathcal{V}}= D_X Y$ because the leaves are assumed to be totally geodesic. Finally,  it is easy to check that for every $\varepsilon >0$, the Bott connection satisfies $\nabla g_\varepsilon=0$.

\begin{Example}
Let $(\M, \theta,g)$ be a K-contact Riemannian manifold. The Bott connection coincides  with the Tanno's connection that was introduced in \cite{Tanno} and which is the unique connection that satisfies:
\begin{enumerate}
\item $\nabla\theta=0$;
\item $\nabla T=0$;
\item $\nabla g=0$;
\item ${T}(X,Y)=d\theta(X,Y)T$ for any $X,Y\in \Gamma^\infty(\mathcal{H})$;
\item ${T}(T,X)=0$ for any vector field $X\in \Gamma^\infty(\mathcal{H})$.
\end{enumerate}
\end{Example}

\

We now introduce some tensors and definitions that will play an important role in the sequel.

\

For $Z \in \Gamma^\infty(T\M)$, there is a  unique skew-symmetric endomorphism  $J_Z:\mathcal{H}_x \to \mathcal{H}_x$ such that for all horizontal vector fields $X$ and $Y$,
\begin{align}\label{Jmap}
g_\mathcal{H} (J_Z (X),Y)= g_\mathcal{V} (Z,T(X,Y)).
\end{align}
where $T$ is the torsion tensor of $\nabla$. We then extend $J_{Z}$ to be 0 on  $\mathcal{V}_x$.  If $Z_1,\cdots,Z_m$ is a local vertical frame, the operator $\sum_{\ee=1}^m J_{Z_\ee}J_{Z_\ee}$ does not depend on the choice of the frame and shall concisely be denoted by $\mathbf{J}^2$. For instance, if $\M$ is a K-contact manifold equipped with the Reeb foliation, then $\mathbf{J}$ is an almost complex structure, $\mathbf{J}^2=-\mathbf{Id}_{\mathcal{H}}$.

\

 The horizontal divergence of the torsion $T$ is the $(1,1)$ tensor  which is defined in a local horizontal frame $X_1,\cdots,X_n$ by
\[
\delta_\mathcal{H} T (X)=- \sum_{j=1}^n(\nabla_{X_j} T) (X_j,X), \quad X \in \Gamma^\infty(\M).
\]
The $g$-adjoint of $\delta_\mathcal{H}T$ will be denoted $\delta_\mathcal{H} T^*$. 

\begin{definition}
We  say that the Riemannian foliation is of Yang-Mills type if $\delta_\mathcal{H} T=0$.
\end{definition}

\begin{Example}
Let $(\M, \theta,g)$ be a K-contact Riemannian manifold. It is easy to see that the Reeb foliation is of Yang-Mills type if and only if $\delta_\Ho d \theta=0$. Equivalently this condition writes $\delta_\Ho J =0$. If $\M$ is a strongly pseudo convex CR manifold with pseudo-Hermitian form $\theta$, then the Tanno's connection is the Tanaka-Webster connection. In that case, we have then $\nabla J=0$ (see \cite{Dragomir}) and thus  $\delta_\Ho J =0$. CR manifold of K-contact type are called Sasakian manifolds (see \cite{Dragomir}). Thus the Reeb foliation on any Sasakian manifold is of Yang-Mills type.
\end{Example}

\begin{Example}
Let $(\M,g)$ be a smooth Riemannian manifold.We endow the tangent bundle $T\M$ with the Sasaki metric so that the bundle projection $\pi: T\M \to \M$ is a Riemannian submersion with totally geodesic fibers. In that case  the torsion of the Bott connection is given by
\[
T(X,Y)=R(X,Y), \quad X,Y \in \Gamma( \Ho),
\]
where $R$ is the curvature of the connection form. By using the second Bianchi identity, the Yang-Mills condition  is equivalent to the fact that the Ricci tensor of the connection form is a Codazzi tensor, that is for any vector fields $X,Y,Z$ in $ \Gamma^\infty ( \Ho)$,
\[
(\nabla_X \text{Ric}) (Y, Z)=(\nabla_Y \text{Ric}) (X, Z).
\]
\end{Example}

\

In the sequel, we shall need to perform computations on one-forms. For that purpose we introduce some definitions and notations on the cotangent bundle.

We say that a one-form to be horizontal (resp. vertical) if it vanishes on the vertical bundle $\mathcal{V}$ (resp. on the horizontal bundle $\mathcal{H}$). We thus have a splitting of the cotangent space
 \[
 T^*_x \bM= \mathcal{H}^*(x) \oplus \mathcal{V}^*(x)
 \]

 The metric $g_\varepsilon$ induces  then a metric on the cotangent bundle which we still denote $g_\varepsilon$. By using similar notations and conventions as before we have for every $\eta$ in $T^*_x \M$,
\[
\| \eta \|^2_{\varepsilon} =\| \eta \|_\mathcal{H}^2+\varepsilon \| \eta \|_\mathcal{V}^2.
\]

\

By using the duality given by the metric $g$, $(1,1)$ tensors can also be seen as linear maps on the cotangent bundle $T^* \M$. More precisely, if $A$ is a $(1,1)$ tensor, we will still denote by $A$ the fiberwise linear map on the cotangent bundle which is defined as the $g$-adjoint of the dual map of $A$. The same convention will be made for any $(r,s)$ tensor.

\

We define then the horizontal Ricci curvature $\mathfrak{Ric}_{\mathcal{H}}$ as the fiberwise symmetric linear map on one-forms such that for every smooth functions $f,g$,
\[
\langle  \mathfrak{Ric}_{\mathcal{H}} (df), dg \rangle=\mathbf{Ricci} (\nabla_\mathcal{H} f ,\nabla_\mathcal{H} g),
\]
where $\mathbf{Ricci}$ is the Ricci curvature of the connection $\nabla$.

\

A simple computation (see for instance Theorem 9.70, Chapter 9 in \cite{Besse}) gives the following result for the Riemannian Ricci curvature of the metric $g_\varepsilon$.

\begin{lemma}\label{Ric}
Assume that the foliation is of Yang-Mills type. Let us denote by $\mathbf{Ricci}_\varepsilon$ the Ricci curvature tensor of the Levi-Civita connection of the metric $g_\varepsilon$ and by $\mathbf{Ricci}_\mathcal{V} $ the Ricci curvature of the leaves, then for every $X \in \Gamma^\infty(\mathcal{H})$ and $Z \in \Gamma^\infty(\mathcal{V})$,
\[
\mathbf{Ricci}_\varepsilon (Z,Z)=\mathbf{Ricci}_\mathcal{V} (Z,Z)+\frac{1}{4\varepsilon^2}\mathbf{Tr} ( J_Z^* J_Z)
\]
\[
\mathbf{Ricci}_\varepsilon (X,Z)=0
\]
\[
\mathbf{Ricci}_\varepsilon (X,X)=\mathbf{Ricci}_\mathcal{H} (X,X)-\frac{1}{2\varepsilon} \| \mathbf{J} X \|^2.
\]
\end{lemma}

We explicitly note that $\mathbf{Ricci}_\varepsilon (X,Z)=0$ is due to the fact that the foliation is assumed to be of Yang-Mills type.

\

If $V$ is a horizontal vector field and $\varepsilon >0$, we consider the fiberwise linear map from the space of one-forms into itself which is given for $\eta \in \Gamma^\infty(T^* \M)$ and $Y \in  \Gamma^\infty(T \M)$ by
\[
\mathfrak{T}^\varepsilon_V \eta (Y) =
\begin{cases}
\frac{1}{\varepsilon} \eta (J_Y V), \quad Y \in \Gamma^\infty(\mathcal{V}) \\
-\eta (T(V,Y)), Y  \in \Gamma^\infty(\mathcal{H})
\end{cases}
\]
We observe that $\mathfrak{T}^\varepsilon_V$ is skew-symmetric for the metric $g_\varepsilon$ so that $\nabla -\mathfrak{T}^\varepsilon$ is a $g_\varepsilon$-metric connection.

If $\eta$ is a one-form, we define the horizontal gradient  of $\eta$ in a local frame  as the $(0,2)$ tensor
\[
\nabla_\mathcal{H} \eta =\sum_{i=1}^n \nabla_{X_i} \eta \otimes \theta_i.
\]
We denote by $\nabla_\ch^\# \eta$ the symmetrization of $\nabla_\mathcal{H} \eta $.

\

Similarly, we will use the notation
\[
\mathfrak{T}^\varepsilon_\mathcal{H} \eta =\sum_{i=1}^n \mathfrak{T}^\varepsilon_{X_i} \eta  \otimes \theta_i.
\]

\

Finally, we will still denote by $\Dh$ the covariant extension on one-forms of the horizontal Laplacian. In a local horizontal frame, we have thus
\[
\Dh=-\nabla_\Ho^* \nabla_\Ho=\sum_{i=1}^n \nabla_{X_i}\nabla_{X_i} -\nabla_{\nabla_{X_i} X_i}.
\]

\subsection{Bochner-Weitzenb\"ock formulas for the horizontal Laplacian}

For $\varepsilon >0$, we consider the following operator which is defined on one-forms by
\[
\square_\varepsilon=-(\nabla_\mathcal{H} -\mathfrak{T}_\mathcal{H}^\varepsilon)^* (\nabla_\mathcal{H} -\mathfrak{T}_\mathcal{H}^\varepsilon)-\frac{1}{ \varepsilon}\mathbf{J}^2+\frac{1}{\varepsilon} \delta_\mathcal{H} T- \mathfrak{Ric}_{\mathcal{H}},
\]
where the adjoint is understood with respect to the metric $g_{\varepsilon}$. It is easily seen that, in a local horizontal frame,
\begin{align}\label{eq-L-form}
-(\nabla_\mathcal{H} -\mathfrak{T}_\mathcal{H}^\varepsilon)^* (\nabla_\mathcal{H} -\mathfrak{T}_\mathcal{H}^\varepsilon)
=\sum_{i=1}^n (\nabla_{X_i} -\mathfrak{T}^\varepsilon_{X_i})^2 - ( \nabla_{\nabla_{X_i} X_i}-  \mathfrak{T}^\varepsilon_{\nabla_{X_i} X_i}),
\end{align}
Observe that if the foliation is of Yang-Mills type then 
\[
\square_\varepsilon=-(\nabla_\mathcal{H} -\mathfrak{T}_\mathcal{H}^\varepsilon)^* (\nabla_\mathcal{H} -\mathfrak{T}_\mathcal{H}^\varepsilon)-\frac{1}{ \varepsilon}\mathbf{J}^2- \mathfrak{Ric}_{\mathcal{H}}.
\]
As a consequence, in the Yang-Mills case the operator $\square_\varepsilon$ is seen to be symmetric for the metric $g_\varepsilon$.

The following theorem that was proved in \cite{BKW} is the main result of the section:

\begin{theorem} \label{Bochner}
For every $f \in C^\infty(\M)$, we have
\[
d \Dh f=\square_\varepsilon df.
\]
\end{theorem}

\begin{proof}
We only sketch the proof and refer to \cite{BKW} for the details. If $Z_1,\cdots,Z_m$ is a local vertical frame of the leaves, we denote
\[
\J(\eta)=-\sum_{\ee=1}^mJ_{Z_\ee}(\iota_{Z_\ee}d\eta_\V),
\]
 where $\eta_\V$ is the the projection of $\eta$ to the vertical cotangent bundle. It does not depend on the choice of the frame and therefore defines a globally defined tensor.
Also, let us consider the map $\mathcal{T} \colon \Gamma^\infty(\wedge^2 T^*\M)\to \Gamma^\infty( T^*\M)$ which is given in a  local coframe $\theta_i \in \Gamma^\infty(\mathcal{H}^*)$, $\nu_k \in \Gamma^\infty(\mathcal{V}^*)$
\[
\mathcal{T}(\theta_i\wedge\theta_j)=-\gamma_{ij}^\ee\nu_\ee,\quad \mathcal{T}(\theta_i\wedge\nu_k)=\mathcal{T}(\nu_k\wedge\nu_\ee)=0.
\]
A direct computation shows then that
\[
 -(\nabla_\mathcal{H} -\mathfrak{T}_\mathcal{H}^\varepsilon)^* (\nabla_\mathcal{H} -\mathfrak{T}_\mathcal{H}^\varepsilon)
=
\Dh +2\J-\frac{2}{\ep}\mathcal{T}\circ d+\delta_\ch T^*-\frac{1}{\ep}\delta_\ch T+\frac{1}{\ep}\mathbf{J}^2.
\]
Thus, we just need to prove that if  $\square_\infty$ is the operator defined  on one-forms by
\[
\square_\infty=\Dh+2\J-\Ric_\ch+\delta_\mathcal{H} T^* ,
\]
then for any $f\in C^\infty(\M)$,
\[
 d\Dh f=\square_\infty df.
\]
A computation in local frame shows that
\begin{align*}
 d\Dh f - \Dh df =  2\J(df) -\Ric_\ch(df) +\delta_\mathcal{H} T^* (df),
\end{align*}
which completes the proof.

\end{proof}

We now state the following Bochner's type identity.

\begin{theorem}\label{Bochner2}
For any   $\eta \in \Gamma^\infty(T^* \M)$,
\[
\frac{1}{2} \Dh \| \eta \|_{\varepsilon}^2 -\langle \square_\varepsilon \eta , \eta \rangle_{\varepsilon} =  \| \nabla_{\mathcal{H}} \eta  -\mathfrak{T}^\varepsilon_{\mathcal{H}} \eta \|_{\varepsilon}^2 + \left\langle \mathfrak{Ric}_{\mathcal{H}} (\eta), \eta \right\rangle_\mathcal{H} -\left \langle \delta_\mathcal{H} T (\eta) , \eta \right\rangle_\mathcal{V} +\frac{1}{\varepsilon} \langle \mathbf{J}^2 (\eta) , \eta \rangle_\mathcal{H}.
\]
\end{theorem}

\begin{proof}
From the very definition of $\square_\varepsilon$, we have
\[
-\langle \square_\varepsilon \eta , \eta \rangle_\varepsilon=\langle (\nabla_\mathcal{H} -\mathfrak{T}_\mathcal{H}^\varepsilon)^* (\nabla_\mathcal{H} -\mathfrak{T}_\mathcal{H}^\varepsilon) \eta,  , \eta \rangle_\varepsilon+ \left\langle \mathfrak{Ric}_{\mathcal{H}} (\eta), \eta \right\rangle_\mathcal{H} -\left \langle \delta_\mathcal{H} T (\eta) , \eta \right\rangle_\mathcal{V} +\frac{1}{\varepsilon} \langle \mathbf{J}^2 (\eta) , \eta \rangle_\mathcal{H},
\]
The idea is now to multiply this by any $g \in C_0^\infty(\M)$ and integrate over $\M$. For that, observe that
\begin{align*}
\int_\M g \langle (\nabla_\mathcal{H} -\mathfrak{T}_\mathcal{H}^\varepsilon)^* (\nabla_\mathcal{H} -\mathfrak{T}_\mathcal{H}^\varepsilon) \eta,  , \eta \rangle_\varepsilon d\mu&
= \int_\M  \langle (\nabla_\mathcal{H} -\mathfrak{T}_\mathcal{H}^\varepsilon) \eta,  ,  (\nabla_\mathcal{H} -\mathfrak{T}_\mathcal{H}^\varepsilon) (g \eta) \rangle_\varepsilon d\mu.
\end{align*}
We have now
\[
 (\nabla_\mathcal{H} -\mathfrak{T}_\mathcal{H}^\varepsilon) (g \eta)=g(\nabla_\mathcal{H} -\mathfrak{T}_\mathcal{H}^\varepsilon) ( \eta) +\eta \otimes \nabla_\mathcal{H} g 
 \]
 and 
 \begin{align*}
 \int_\M  \langle (\nabla_\mathcal{H} -\mathfrak{T}_\mathcal{H}^\varepsilon) \eta,  ,  \eta \otimes \nabla_\mathcal{H} g  \rangle_\varepsilon d\mu &=\int_\M  \langle \nabla_\mathcal{H} \eta,  ,  \eta \otimes \nabla_\mathcal{H} g  \rangle_\varepsilon d\mu \\
  &=\frac{1}{2} \int_\M  \langle \nabla_\mathcal{H} g,  ,  \nabla_\mathcal{H} \| \eta \|^2 \rangle_\varepsilon d\mu.
 \end{align*}
 Putting things together we deduce that
 \[
 \int_\M g \langle (\nabla_\mathcal{H} -\mathfrak{T}_\mathcal{H}^\varepsilon)^* (\nabla_\mathcal{H} -\mathfrak{T}_\mathcal{H}^\varepsilon) \eta,  , \eta \rangle_\varepsilon d\mu
 =\int_\M g  \| \nabla_{\mathcal{H}} \eta  -\mathfrak{T}^\varepsilon_{\mathcal{H}} \eta \|_{\varepsilon}^2 d\mu-\frac{1}{2}  \int_\M g \Dh \| \eta \|_{\varepsilon}^2 d\mu.
 \]
 Since it is true for every $g$, we deduce
 \[
 \langle (\nabla_\mathcal{H} -\mathfrak{T}_\mathcal{H}^\varepsilon)^* (\nabla_\mathcal{H} -\mathfrak{T}_\mathcal{H}^\varepsilon) \eta,  , \eta \rangle_\varepsilon=
 \| \nabla_{\mathcal{H}} \eta  -\mathfrak{T}^\varepsilon_{\mathcal{H}} \eta \|_{\varepsilon}^2 -\frac{1}{2}  \Dh \| \eta \|_{\varepsilon}^2.
 \]
\end{proof}

Let us observe that if $\eta=df$ for some $f \in C^\infty(\M)$, than an easy computation shows that
\[
 \| \nabla_{\mathcal{H}} \eta  -\mathfrak{T}^\varepsilon_{\mathcal{H}} \eta \|_{\varepsilon}^2
   = \| \nabla^\#_{\mathcal{H}} \eta  \|_{\varepsilon}^2
  -\frac{1}{4}\mathbf{Tr}_\mathcal{H} (J^2_{\eta})
  +\ep \| \nabla_{\mathcal{H}} \eta  -\mathfrak{T}^\varepsilon_{\mathcal{H}} \eta \|_{\mathcal{V}}^2 ,
  \]
  thus by Cauchy-Schwarz inequality we have,
\begin{align}\label{bochner3}
 & \frac{1}{2} \Dh \| \eta \|_{\varepsilon}^2 -\langle \square_\varepsilon \eta , \eta \rangle_{\varepsilon}  \\
   \ge & \frac{1}{n}\left( \mathbf{Tr}_\mathcal{H}  \nabla_\ch^\# \eta \right)^2 -\frac{1}{4} \mathbf{Tr}_\mathcal{H} (J^2_{\eta})+ \left\langle \mathfrak{Ric}_{\mathcal{H}} (\eta), \eta \right\rangle_\mathcal{H} -\left \langle \delta_\mathcal{H} T (\eta) , \eta \right\rangle_\mathcal{V} +\frac{1}{\varepsilon} \langle \mathbf{J}^2 (\eta) , \eta \rangle_\mathcal{H} \notag
\end{align}

\subsection{Generalized curvature dimension inequality}

Let $\M$ be a smooth, connected  manifold with dimension $n+m$. We assume that $\bM$ is equipped with a Riemannian foliation $\mathcal{F}$ with bundle like metric $g$ and totally geodesic  $m$-dimensional leaves for which the horizontal distribution is Yang-Mills. We also assume that $\M$ is complete and that globally on $\M$, for every $\eta_1 \in \Gamma^\infty(\mathcal{H}^*)$ and $\eta_2 \in \Gamma^\infty(\mathcal{V}^*)$,
\[
\langle \mathfrak{Ric}_{\mathcal{H}} (\eta_1), \eta_1 \rangle_{\mathcal{H}} \ge  \rho_1 \| \eta_1 \|^2_{\mathcal{H}}, \quad -\langle \mathbf{J}^2 \eta_1, \eta_1  \rangle_\mathcal{H} \le \kappa  \| \eta_1 \|^2_\mathcal{H},\quad -\frac{1}{4} \mathbf{Tr}_\mathcal{H} (J^2_{\eta_2})\ge \rho_2 \| \eta_2 \|^2_\mathcal{V},
\]
for some $\rho_1 \in \mathbb{R}$, $\kappa,\rho_2 >0$. The third assumption can be thought as a uniform bracket generating condition of the horizontal distribution $\mathcal{H}$ and from H\"ormander's theorem, it implies that the horizontal Laplacian $\Dh$ is a subelliptic diffusion operator. We insist that for the following results below to be true, the positivity of $\rho_2$ is required. 

\

We introduce the following operators defined for $f,g \in C^\infty(\M)$,
\[
\Gamma(f,g)=\frac{1}{2} ( \Dh (fg) -g\Dh f-f\Dh g)=\langle \nabla_\mathcal{H} f , \nabla_\mathcal{H} g\rangle_\mathcal{H}
\]

\[
\Gamma^\mathcal{V} (f,g)=\langle \nabla_\mathcal{V} f , \nabla_\mathcal{V} g\rangle_\mathcal{V}
\]
and their iterations which are defined by
\[
\Gamma_2(f,g)=\frac{1}{2} ( \Dh(\Gamma(f,g)) -\Gamma(g,\Dh f)-\Gamma(f,\Dh g))
\]
\[
\Gamma^\mathcal{V}_2(f,g)=\frac{1}{2} ( \Dh (\Gamma^\mathcal{V}(f,g)) -\Gamma^\mathcal{V}(g,\Dh f)-\Gamma^\mathcal{V}(f,\Dh g))
\]
As a consequence of Theorem \ref{Bochner}, we obtain the curvature dimension inequality introduced with Nicola Garofalo in \cite{BG}.
\begin{theorem}\label{CD}
For every $f,g \in C^\infty(\M)$, and $\varepsilon >0$,
\[
\Gamma_2(f,f)+\varepsilon \Gamma^\mathcal{V}_2(f,f)\ge \frac{1}{n} (\Dh f)^2 +\left( \rho_1 -\frac{\kappa}{\varepsilon}\right) \Gamma(f,f)+\rho_2 \Gamma^\mathcal{V} (f,f),
\]
and
\[
\Gamma (f, \Gamma^\mathcal{V} (f))=\Gamma^\mathcal{V}  (f, \Gamma (f)).
\]
\end{theorem}
\begin{proof}
From the inequality \eqref{bochner3}, we have for every   $\eta=df \in \Gamma^\infty(T^* \M)$,
\begin{align*}
 \frac{1}{2} \Dh \| \eta \|_{\varepsilon}^2 -\langle \square_\varepsilon \eta , \eta \rangle_{\varepsilon}    \ge & \frac{1}{n}\left( \mathbf{Tr}_\mathcal{H}  \nabla_\ch^\# \eta \right)^2 -\frac{1}{4} \mathbf{Tr}_\mathcal{H} (J^2_{\eta})+ \left\langle \mathfrak{Ric}_{\mathcal{H}} (\eta), \eta \right\rangle_\mathcal{H}  +\frac{1}{\varepsilon} \langle \mathbf{J}^2 (\eta) , \eta \rangle_\mathcal{H}.
\end{align*}
Using this inequality and taking into account the assumptions
\[
\langle \mathfrak{Ric}_{\mathcal{H}} (\eta_1), \eta_1 \rangle_{\mathcal{H}} \ge  \rho_1 \| \eta_1 \|^2_{\mathcal{H}}, \quad -\langle \mathbf{J}^2 \eta_1, \eta_1  \rangle_\mathcal{H} \le \kappa  \| \eta_1 \|^2_\mathcal{H},\quad -\frac{1}{4} \mathbf{Tr}_\mathcal{H} (J^2_{\eta_2})\ge \rho_2 \| \eta_2 \|^2_\mathcal{V},
\]
immediately yields the expected result. The intertwining $\Gamma (f, \Gamma^\mathcal{V} (f))=\Gamma^\mathcal{V}  (f, \Gamma (f))$ is proved in Theorem \ref{inter_commutation}.
\end{proof}

\subsection{Sharp lower bound for the first eigenvalue of the horizontal Laplacian}

In this section, as a second application of the transverse Weitzenb\"ock formula proved in the previous chapter, we obtain a sharp lower for the first non zero eigenvalue of the horizontal Laplacian.  

Let $\M$ be a compact,  smooth, connected  manifold with dimension $n+m$. We assume that $\bM$ is equipped with a Riemannian foliation $\mathcal{F}$ with bundle like metric $g$ and totally geodesic  $m$-dimensional leaves. We also assume that $\M$ is of Yang-Mills type.

\

We prove the following result that was first obtained in \cite{BK2} in a less general setting. Let us point out that this bound may not be obtained as a consequence of the generalized curvature dimension inequality only.

\begin{theorem}\label{lichne}
Assume that for every smooth horizontal one-form $\eta$,
\[
\langle \mathfrak{Ric}_{\mathcal{H}}(\eta),\eta \rangle_\mathcal{H} \ge \rho_1 \| \eta \|^2_\mathcal{H},\quad  \left\langle -\mathbf{J}^2(\eta), \eta \right\rangle_\mathcal{H} \le \kappa \| \eta \|^2_\mathcal{H},
\]
and that for every vertical one-form $\eta$,
\[
\mathbf{Tr} ( J_\eta^* J_\eta) \ge \rho_2 \| \eta \|_\mathcal{V}^2,
\]
with $\rho_1,\rho_2 >0$ and $\kappa \ge 0$. Then the first eigenvalue $\lambda_1$ of the horizontal Laplacian $-\Dh$  satisfies
\begin{align*}
 \lambda_1 \geq  \frac{\rho_1}{1-\frac{1}{n}+\frac{3\kappa}{\rho_2}}.
\end{align*}
\end{theorem}
To put things in perspective, we give  examples where this bound is sharp.
\begin{itemize}
\item Let us consider   the Hopf fibration $ \mathbf{U}(1) \to \mathbb{S}^{2d+1} \to \mathbb{CP}^d$. As we know, the horizontal Laplacian $\Dh$ is  the lift of the Laplace-Beltrami operator on $\mathbb{CP}^d$ and in that case $\lambda_1=2d$ (see \ref{spectreS1}). On the other hand, for this example, $\rho_1=2(d+1)$, $\kappa=1$, $\rho_2=2d$.  Thus the bound of Theorem \ref{lichne} is sharp.

\item Consider now the quaternionic Hopf fibration $ \mathbf{SU}(2) \to \mathbb{S}^{4d+3} \to \mathbb{HP}^d$. The sub-Laplacian $\Dh $ is then the lift of the Laplace-Beltrami operator on  $\mathbb{HP}^d$ and in that case, $\lambda_1=4d$ (see \ref{spectreSU}).  For this example, $\rho_1=4(d+2)$, $\kappa=3$, $\rho_2=4d$.  Thus the bound of Theorem \ref{lichne} is still sharp in this example.

\end{itemize}
We also mention that it has even been proved in \cite{BK2} that for some Riemannian foliations the equality  $\lambda_1 =  \frac{\rho_1}{1-\frac{1}{n}+\frac{3\kappa}{\rho_2}}$ actually implies that the foliation is equivalent to the classical or the quaternionic Hopf fibration.

\begin{proof}
As for the classical Lichnerowicz estimate on Riemannian manifolds, the idea is to integrate on the manifold the Bochner-Weitzenb\"ock equality in Theorem \ref{Bochner2} but some tricks have to be done.  
Let $f \in C^\infty(\M)$. Let us first observe that
\begin{align*}
-\int_\M \langle \square_\varepsilon df, df \rangle_\varepsilon d\mu&= -\int_\M \langle  d\Dh f, df \rangle_\varepsilon d\mu \\
 &= -\int_\M \langle  d\Dh f, df \rangle_\Ho d\mu-\varepsilon  \int_\M \langle  d\Dh f, df \rangle_\V d\mu
\end{align*}

Thus, by integrating the  Bochner-Weitzenb\"ock equality in Theorem \ref{Bochner2}, we obtain
\begin{align}\label{BB}
\int_\M (\Dh f)^2 d\mu -\varepsilon \int_\M \langle  d(\Dh f) , df \rangle_{\mathcal{V}} d\mu  \ge \int_\M \| \nabla_{\mathcal{H}} df  -\mathfrak{T}_\mathcal{H}^\varepsilon  df \|_{\varepsilon}^2 d\mu +\left( \rho_1-\frac{\kappa}{\varepsilon}\right)\int_\M \| df \|_{\mathcal{H}}^2 d\mu .
\end{align}
We now compute
\begin{align}\label{IPP2}
 & \int_\M \| \nabla_{\mathcal{H}} df  -\mathfrak{T}_\mathcal{H}^\varepsilon  df \|_{\varepsilon}^2 d\mu \notag  \\
 =&\int_\M \| \nabla_{\mathcal{H}} df  -\mathfrak{T}_\mathcal{H}^\varepsilon  df \|_{\mathcal{H}}^2 d\mu+ \varepsilon \int_\M \| \nabla_{\mathcal{H}} df  -\mathfrak{T}_\mathcal{H}^\varepsilon  df \|_{\mathcal{V}}^2 d\mu \notag  \\
 =&\int_\M \| \nabla_{\mathcal{H}} df  -\mathfrak{T}_\mathcal{H}^\varepsilon  df \|_{\mathcal{H}}^2 d\mu+ \varepsilon \int_\M \| \nabla_{\mathcal{H}} df   \|_{\mathcal{V}}^2 d\mu -2 \varepsilon \int_\M \langle \nabla_\mathcal{H} df , \mathfrak{T}_\mathcal{H}^\varepsilon (df) \rangle_{\mathcal{V}} d\mu+ \varepsilon \int_\M \| \mathfrak{T}_\mathcal{H}^\varepsilon df   \|_{\mathcal{V}}^2d\mu.
\end{align}
Using the definition $\mathfrak{T}_\mathcal{H}^\varepsilon $ together with the Yang-Mills assumption, we see that
\begin{align}\label{IPP}
\int_\M \langle \nabla_\mathcal{H} df , \mathfrak{T}_\mathcal{H}^\varepsilon (df) \rangle_{\mathcal{V}}d\mu=\frac{1}{ \varepsilon} \int_\M \mathbf{Tr} ( J^*_{\nabla_\mathcal{V}f} J_{\nabla_\mathcal{V}f})d\mu.
\end{align}
By using \eqref{IPP}, the trick is now to write
\begin{align*}
\int_\M \langle \nabla_\mathcal{H} df , \mathfrak{T}_\mathcal{H}^\varepsilon (df) \rangle_{\mathcal{V}}d\mu&=\frac{3}{2} \int_\M \langle \nabla_\mathcal{H} df , \mathfrak{T}_\mathcal{H}^\varepsilon (df) \rangle_{\mathcal{V}}d\mu -\frac{1}{2} \int_\M \langle \nabla_\mathcal{H} df , \mathfrak{T}_\mathcal{H}^\varepsilon (df) \rangle_{\mathcal{V}} d\mu\\
 &=\frac{3}{2} \int_\M \langle \nabla_\mathcal{H} df , \mathfrak{T}_\mathcal{H}^\varepsilon (df) \rangle_{\mathcal{V}} d\mu- \frac{1}{4 \varepsilon} \int_\M \mathbf{Tr} ( J^*_{\nabla_\mathcal{V}f} J_{\nabla_\mathcal{V}f})d\mu.
\end{align*}
Coming back to \eqref{IPP2} and completing the squares gives
\begin{align*}
\int_\M \| \nabla_{\mathcal{H}} df  -\mathfrak{T}_\mathcal{H}^\varepsilon  df \|_{\varepsilon}^2d\mu =& \int_\M \| \nabla_{\mathcal{H}} df  -\mathfrak{T}_\mathcal{H}^\varepsilon  df \|_{\mathcal{H}}^2 d\mu+\varepsilon \int_\M \left\| \nabla_{\mathcal{H}} df  -\frac{3}{2}\mathfrak{T}_\mathcal{H}^\varepsilon  df \right\|_{\mathcal{V}}^2d\mu \\
 & +\frac{1}{2} \int_\M \mathbf{Tr} ( J^*_{\nabla_\mathcal{V}f} J_{\nabla_\mathcal{V}f})d\mu-\frac{5}{4}\varepsilon \int_\M \|\mathfrak{T}_\mathcal{H}^\varepsilon  df \|_{\mathcal{V}}^2d\mu.
 \end{align*}
 This yields the lower bound
\[
\int_\M \| \nabla_{\mathcal{H}} df  -\mathfrak{T}_\mathcal{H}^\varepsilon  df \|_{\varepsilon}^2d\mu \ge \frac{1}{n}\int_\M  (\Dh f)^2 d\mu+\frac{3}{4} \rho_2 \int_\M  \| df \|_{\mathcal{V}}^2d\mu-\frac{5}{4 \varepsilon} \kappa \int_\M  \| df \|_{\mathcal{H}}^2 d\mu.
\]
We thus deduce
\[
\frac{n-1}{n} \int_\M (\Dh f)^2d\mu -\varepsilon \int_\M \langle  d(\Dh f) , df \rangle_{\mathcal{V}}d\mu   \ge \left( \rho_1-\frac{9\kappa}{4\varepsilon}\right)\int_\M \| df \|_{\mathcal{H}}^2d\mu +\frac{3}{4} \rho_2 \int_\M  \| df \|_{\mathcal{V}}^2d\mu .
\]
Now if $f$ is an eigenfunction that satisfies $\Dh f =-\lambda_1 f$, we get
\[
\frac{n-1}{n} \lambda_1^2 \int_\M f^2 d\mu +\varepsilon \lambda_1 \int_\M  \| df \|_{\mathcal{V}}^2d\mu   \ge \left( \rho_1-\frac{9\kappa}{4\varepsilon}\right)\lambda_1 \int_\M f^2d\mu +\frac{3}{4} \rho_2 \int_\M  \| df \|_{\mathcal{V}}^2d\mu .
\]
Choosing $\varepsilon$ such that
\[
\varepsilon \lambda_1 =\frac{3}{4} \rho_2,
\]
yields the desired lower bound on $\lambda_1$.
\end{proof}

\section{The  horizontal heat semigroup}

We introduce here a fundamental tool in the geometric analysis of Riemannian foliations: the horizontal heat semigroup. We study then some of its properties like stochastic completeness and quickly discuss the Li-Yau estimates for this semigroup. 

\subsection{Essential self-adjointness of the horizontal Laplacian}

Let $\M$ be a smooth, connected  manifold with dimension $n+m$. We assume that $\bM$ is equipped with a Riemannian foliation  with a bundle like metric $g$ and totally geodesic  $m$-dimensional leaves. We assume that the metric $g$ is complete and denote by $C_0^\infty(\M)$ the space of smooth and compactly supported functions on $\M$. We will also assume that the horizontal distribution $\mathcal{H}$ of the foliation is bracket generating. From H\"ormander's theorem, the bracket generating condition implies that the horizontal Laplacian $\Dh$ is hypoelliptic.

\

An important consequence of the completeness assumption is  the fact that there exists an increasing
sequence $h_n\in C^\infty_0(\bM)$  such that $h_n\nearrow 1$ on
$\bM$, and 
\begin{align}\label{exhaustion}
||\nabla_{\mathcal{H}} h_n||_{\infty} + ||\nabla_{\mathcal{V}} h_n||_{\infty} \to 0,
\end{align}
as $n\to \infty$. We refer to Strichartz (\cite{Strichartz}) for a proof of this fact.

It will be convenient to introduce the following operators defined for $f,g \in C^\infty(\M)$ by
\[
\Gamma(f,g)=\frac{1}{2} ( \Dh(fg) -g\Dh f -f\Dh g)=\langle \nabla_\mathcal{H} f , \nabla_\mathcal{H} g\rangle_\mathcal{H}
\]
and 
\[
\Gamma^\mathcal{V} (f,g)=\langle \nabla_\mathcal{V} f , \nabla_\mathcal{V} g\rangle_\mathcal{V}.
\]
As a shorthand notation, we will use the notations $\Gamma(f)=\Gamma(f,f)$ and $\Gamma^\mathcal{V} (f)=\Gamma^\mathcal{V} (f,f)$.

\begin{proposition}
The horizontal Laplacian $\Dh$ is essentially self-adjoint on the space $C_0^\infty(\M)$. 
\end{proposition}

\begin{proof}
According to Reed-Simon \cite{reed1}, p. 137, it is enough
to prove that  if $\Dh^* f=\lambda f$ with $\lambda >0$, then
$f=0$. Since $\Dh$ is given on the domain $C_0^\infty(\M)$, this means that $\Dh f=\lambda f$ in the sense of distributions.

From the hypoellipticity of $\Dh$, we first deduce that $f$ has
to be a smooth function. Now, for $h \in C^\infty_0(\bM)$,

\begin{align*}
\int_\M \Gamma( f, h^2f) d\mu=-\int_\M  f\Dh (h^2f) d\mu=-\int_\M  (\Dh^*f)(h^2f) d\mu =-\lambda \int_\M  f^2h^2 d\mu \le 0.
\end{align*}

Since
\[
 \Gamma( f, h^2f)=h^2 \Gamma (f,f)+2 fh \Gamma(f,h),
 \]
 we deduce that
 \[
 \int_\M  h^2\Gamma (f) d\mu +2 \int_\M  h \Gamma(f,h) d\mu \le 0.
 \]
 Therefore, by Schwarz inequality
 \[
 \int_\M  h^2\Gamma (f) d\mu \le 4 \| f |_2^2 \| \Gamma (h) \|_\infty.
 \]
 If we now use a sequence $h_n$ that satisfies \ref{exhaustion} and let $n \to \infty$, we obtain $\Gamma(f)=0$ and therefore $f=0$, as desired.
\end{proof}

If $\Dh=-\int_0^{+\infty} \lambda dE_\lambda$ is the spectral
resolution of the Friedrichs extension of  $\Dh$ in $L^2 (\bM,\mu)$, then by definition, the
heat semigroup $(P_t)_{t \ge 0}$ is given by $P_t= \int_0^{+\infty}
e^{-\lambda t} dE_\lambda$. It is a symmetric Markov semigroup on 
$L^2 (\bM,\mu)$. That is, it satisfies the following properties:
\begin{itemize}
\item $P_0=\mathbf{Id}$;
\item $P_{t+s}=P_t P_s$, $s,t \ge 0$;
\item For $f \in L^2(\M,\mu)$, $\lim_{t \to 0}\| P_t f -f \|_2=0$;
\item $\|P_t f \|_2 \le \| f \|_2$;
\item If $f \in L^2(\M,\mu)$ is non negative, then $P_t f \ge 0$;
\item If $f \in L^2(\M,\mu)$ is less than one, then $P_t f \le 1$.
\end{itemize}

By using the Riesz-Thorin interpolation theorem, $(P_t)_{t\ge 0}$ induces a contraction semigroup on all the $L^p(\M,\mu)$'s, $1\le p \le \infty$.

Due to the hypoellipticity of $\Dh$, $(t,x) \rightarrow P_t f(x)$ is
smooth on $\mathbb{M}\times (0,\infty) $ and
\[ P_t f(x)  = \int_{\mathbb M} p(x,y,t) f(y) d\mu(y),\ \ \ f\in
C^\infty_0(\mathbb M),\] where $p(x,y,t) > 0$ is the so-called heat
kernel associated to $P_t$. Such function is smooth and it is symmetric, i.e., \[ p(x,y,t)
= p(y,x,t). \]
 By the semigroup property for every $x,y\in \bM$ and $0<s,t$ we have
\begin{equation*}
p(x,y,t+s) = \int_\bM p(x,z,t) p(z,y,s) d\mu(z) = \int_\bM p(x,z,t)
p(y,z,s) d\mu(z) = P_s(p(x,\cdot,t))(y).
\end{equation*}

For a more analytic view of $(P_t)_{t\ge 0}$, we recall that it can be seen as the unique solution of a parabolic Cauchy problem in $L^p(\bM,\mu),1<p<+\infty$.
\begin{proposition}\label{uniquenessLp}
The unique solution of the Cauchy problem
\[
\begin{cases}
\frac{\p u}{\p t} - \Dh u = 0,
\\
u(x,0) = f(x),\ \ \ \   f\in L^p(\bM,\mu),1<p<+\infty,
\end{cases}
\]
that satisfies $\| u(\cdot,t) \|_p <\infty$ for every $t \ge 0$, is given by $u(x,t)=P_t f(x)$.
\end{proposition}

We stress that without further conditions, this result fails when $p=1$ or $p=+\infty$. The case $p=+\infty$ is equivalent to stochastic completeness ($P_t 1=1$) and will be discussed in a later section.

\subsection{Horizontal heat semigroup on one-forms}

Throughout the section, we work under the same assumptions as the previous section and we moreover assume  that for every horizontal one-form $\eta$,
\[
 \langle \mathfrak{Ric}_{\mathcal{H}} (\eta) , \eta  \rangle_\mathcal{H} \ge -K \| \eta \|^2_\mathcal{H} , \quad -\langle \mathbf{J}^2 \eta, \eta  \rangle_\mathcal{H} \le \kappa  \| \eta \|^2_\mathcal{H},
\]
with $K,\kappa \ge 0$. We also assume that the horizontal distribution $\mathcal{H}$ is Yang-Mills, which means that
\[
\delta_\mathcal{H} T=0.
\]
We recall  that if we consider the operator defined on one-forms by the formula
 \[
\square_\varepsilon=-(\nabla_\mathcal{H} -\mathfrak{T}_\mathcal{H}^\varepsilon)^* (\nabla_\mathcal{H} -\mathfrak{T}_\mathcal{H}^\varepsilon)-\frac{1}{ \varepsilon}\mathbf{J}^2 - \mathfrak{Ric}_{\mathcal{H}},
\]
then for any smooth function $f$,
\[
d\Dh f =\square_{\varepsilon} df
\]
and  for any smooth one-form $\eta$
\begin{align*}
\frac{1}{2} \Dh \| \eta \|_{2\varepsilon}^2 -\langle \square_\varepsilon \eta , \eta \rangle_{\varepsilon} & =  \| \nabla_{\mathcal{H}} \eta  -\mathfrak{T}^\varepsilon_{\mathcal{H}} \eta \|_{\varepsilon}^2 +\left\langle \left(\mathfrak{Ric}_{\mathcal{H}}+\frac{1}{ \varepsilon} \mathbf{J}^2\right)\eta, \eta \right\rangle_{\varepsilon} \\
 & \ge \left( \rho-\frac{\kappa}{\varepsilon} \right) \| \eta \|^2_\mathcal{H} .
\end{align*}

The operator $\square_\varepsilon$ is symmetric for the metric 
\[
g_\varepsilon=g_\mathcal{H} \oplus \frac{1}{\varepsilon} g_\mathcal{V}.
\] 
Thanks to our assumptions we can even say more.

\begin{lemma}
The operator  $\square_\varepsilon$ is essentially self-adjoint on the space of smooth and compactly supported one-forms for the Riemannian metric $g_{\varepsilon}$.
\end{lemma}

\begin{proof}
We consider an increasing sequence $h_n\in C_0^\infty(\M)$, $0 \le h_n \le 1$,  such that $h_n\nearrow 1$ on $\mathbb{M}$, and $||\Gamma (h_n)||_{\infty} \to 0$, as $n\to \infty$.

To prove that $\square_\varepsilon$ is essentially self-adjoint, once again it is enough to prove that for some $\lambda >0$, $\square_\varepsilon \eta =\lambda \eta $ with $\eta \in L^2$ implies $\eta =0$. So, let $\lambda >0$ and $\eta \in  L^2$ such that $\square_\varepsilon \eta =\lambda \eta $. We have then
\begin{align*}
 & \lambda \int_\M h_n^2 \| \eta \|_{\varepsilon}^2 \\
 =& \int_\M  \langle  h_n^2 \eta , \square_\varepsilon \eta \rangle_{ \varepsilon} \\
 =&- \int_\M \langle \nabla_\mathcal{H} (h_n^2 \eta )-\mathfrak{T}_\mathcal{H}^\varepsilon (h_n^2 \eta) , \nabla_\mathcal{H} \eta -\mathfrak{T}_\mathcal{H}^\varepsilon \eta \rangle_{\varepsilon} +\int_\M h_n^2 \left\langle \left(-\frac{1}{ \varepsilon}\mathbf{J}^2 - \mathfrak{Ric}_{\mathcal{H}}\right)(\eta), \eta \right\rangle_{ \varepsilon} \\
 =&-\int_\M h_n^2 \| \nabla_\mathcal{H} \eta -\mathfrak{T}_\mathcal{H}^\varepsilon \eta \|_{\varepsilon}^2 -2\int_\M h_n  \langle \eta, \nabla_{\nabla_\mathcal{H} h_n} \eta \rangle_{\varepsilon} +\int_\M h_n^2 \left\langle \left(-\frac{1}{ \varepsilon}\mathbf{J}^2 - \mathfrak{Ric}_{\mathcal{H}}\right)(\eta), \eta \right\rangle_{ \varepsilon} .
\end{align*}
From our assumptions, the symmetric tensor $-\frac{1}{ \varepsilon}\mathbf{J}^2 - \mathfrak{Ric}_{\mathcal{H}}$ is bounded from above, thus by choosing $\lambda$ big enough, we have
\[
\int_\M h_n^2 \| \nabla_\Ho \eta -\mathfrak{T}_{\Ho}^\varepsilon \eta \|_{\varepsilon}^2 +2\int_\M h_n  \langle \eta, \nabla_{\nabla_\mathcal{H} h_n} \eta \rangle_{\varepsilon} \le 0.
\]
By letting $n\to \infty$, we easily deduce that $\| \nabla_\mathcal{H} \eta -\mathfrak{T}_\mathcal{H}^\varepsilon \eta \|_{\varepsilon}^2=0$ which implies $ \nabla_\mathcal{H} \eta -\mathfrak{T}_\mathcal{H}^\varepsilon \eta=0$. If we come back to the equation $\square_\varepsilon \eta =\lambda \eta $ and the expression of $\square_\varepsilon$, we see that it implies that:
\[
\left( -\frac{1}{ \varepsilon}\mathbf{J}^2 - \mathfrak{Ric}_{\mathcal{H}} \right) (\eta) =\lambda \eta.
\]
Our choice of $\lambda$ forces then $\eta=0$.
\end{proof}

Since $ \square_\varepsilon$ is essentially self-adjoint, it admits a unique self-adjoint extension which generates thanks to the spectral theorem a semigroup $Q^\varepsilon_t=e^{ t \square_\varepsilon}$.   We recall that $P_t=e^{ t \Dh}$ the semigroup generated by $ \Dh$. We have the following commutation property:

\begin{lemma}\label{commu2}
If $f \in C_0^\infty(\M)$, then for every $t \ge 0$,
\[
d P_t f=Q^\varepsilon_t df.
\]
\end{lemma}

\begin{proof}
Let $\eta_t =Q^\varepsilon_t df$. By essential self-adjointness, it is the unique solution in $L^2$ of the heat equation
\[
\frac{\partial \eta}{\partial t}=  \square_\varepsilon \eta,
\]
with initial condition $\eta_0 =df$. From the fact that
\[
dL=\square_\varepsilon d,
\]
we see that $\alpha_t=dP_t f$ solves the heat equation
\[
\frac{\partial \alpha}{\partial t}=  \square_\varepsilon \alpha
\]
with the same initial condition $\alpha_0=df$. In order to conclude, we thus just need to prove that for every $t \ge 0$, $dP_tf$ is in $L^2$ . As usual, we denote by $\Dh$ the vertical  Laplacian. The Laplace-Beltrami operator of $\M$ is therefore $\Delta=\Dh+\Dv$. Since the leaves are totally geodesic, $\Delta$ commutes with $\Dh$ on $C^2$ functions. Moreover from the spectral theorem, $\Dh e^{t \Delta}$ maps $C_0^\infty(\M)$ into $L^2(\M,\mu)$. We deduce by essential self-adjointness that $ \Dh e^{t \Delta}=e^{t \Delta} \Dh$. Similarly we obtain $ e^{s\Dh} e^{t \Delta}=e^{t \Delta} e^{s\Dh}$ which implies $\Delta e^{s\Dh}=e^{s\Dh} \Delta$. As a consequence we have that for every $t \ge 0$, $dP_tf$ is in $L^2$.
\end{proof}

\subsection{Stochastic completeness}

We can now give an important corollary of the commutation of Lemma \ref{commu2}.

\begin{theorem}\label{complete}
For every $\varepsilon >0$, $ t\ge 0, x\in \M $ and $ f \in C_0^\infty(\M)$,
\[
\| dP_t f (x) \|_\varepsilon \le e^{\left( K +\frac{\kappa}{\varepsilon} \right) t} P_t \| d f \|_\varepsilon (x).
\]
\end{theorem}

\begin{proof}
The idea is to use the Feynman-Kac  stochastic representation of $Q_t^{\varepsilon}$. We denote by $(X_t)_{t \ge 0}$ the symmetric diffusion process generated by $\frac{1}{2} L$ and denote by $\mathbf{e}$ its lifetime.
Consider the process $\tau_t^\varepsilon:T^*_{X_t} \M \to T^*_{X_0} \M  $ which is the solution of  the following covariant Stratonovitch stochastic differential equation:
\begin{align}\label{tau}
d \left[ \tau^\varepsilon_t \alpha (X_t)\right]= \tau^\varepsilon_t \left( \nabla_{\circ dX_t}  -\mathfrak{T}^\varepsilon_{\circ dX_t} - \frac{1}{2} \left( \frac{1}{ \varepsilon} \mathbf{J}^2 +\mathfrak{Ric}_{\mathcal{H}} \right) dt \right)\alpha (X_t), \quad \tau^\varepsilon_0=\mathbf{Id},
\end{align}
where $\alpha$ is any smooth one-form. By using Gronwall's lemma, we have for every $ t\ge 0$,
\[
\| \tau^\varepsilon_t \alpha (X_t) \|_{ \varepsilon} \le e^{\frac{1}{2}\left( K+\frac{\kappa}{ \varepsilon} \right)t} \| \alpha (X_t) \|_{ \varepsilon}.
\]
By the Feynman-Kac formula, we have for every  smooth and compactly supported one-form
\[
Q_{t/2} \eta (x)=\mathbb{E}_x \left( \tau_t \eta (X_t) \mathbf{1}_{t < \mathbf{e}} \right).
\]
Since $dP_t=Q_t^\varepsilon d$, it follows easily that
\[
\| dP_t f (x) \|_\varepsilon \le e^{\left( K +\frac{\kappa}{\varepsilon} \right) t} P_t \| d f \|_\varepsilon (x).
\]
\end{proof}

It is well-known that this type of gradient bound implies the stochastic completeness of $P_t$. More precisely, adapting  an argument of Bakry \cite{Bak} yields the following result. 
\begin{theorem}\label{T:sc}
 For $t \ge 0$, one has $ P_t 1 =1$.
\end{theorem}
\begin{proof}
Let $f,g \in  C^\infty_0(\mathbb M)$,  we have
\begin{align*}
\int_{\bM} (P_t f -f) g d\mu & = \int_0^t \int_{\bM}\left(
\frac{\partial}{\partial s} P_s f \right) g d\mu ds \\
 & = \int_0^t
\int_{\bM}\left(\Dh P_s f \right) g d\mu ds \\
 & =- \int_0^t \int_{\bM} \Gamma ( P_s f , g) d\mu ds.
\end{align*}
By means of  Cauchy-Schwarz inequality we
find
\begin{equation}\label{P78}
\left| \int_{\bM} (P_t f -f) g d\mu \right| \le \left(\int_0^t
e^{\left( K +\frac{\kappa}{\varepsilon} \right) s}  ds\right) \sqrt{ \| \Gamma (f) \|_\infty +\varepsilon \|
\Gamma^\V (f) \|_\infty } \int_{\bM}\Gamma (g)^{\frac{1}{2}}d\mu.
\end{equation}
We now apply \eqref{P78} with $f = h_n$, where $h_n$ is an increasing sequence in $ C_0^\infty(\M)$, $0 \le h_n \le 1$,  such that $h_n\nearrow 1$ on $\mathbb{M}$, and $||\Gamma (h_n)||_{\infty} \to 0$, as $n\to \infty$.

By  monotone convergence theorem we have $P_t h_k(x)\nearrow P_t 1(x)$ for every $x\in \bM$. We conclude that the
left-hand side of \eqref{P78} converges to $\int_{\bM} (P_t 1 -1) g d\mu$. Since the right-hand side converges to zero, we reach the conclusion
\[
\int_{\bM} (P_t 1 -1) g d\mu=0,\ \ \ g\in C^\infty_0(\bM).
\]
Since it is true for every $ g\in C^\infty_0(\bM)$, it follows that $P_t 1 =1$.
\end{proof}

It is classical and easy to prove that stochastic completeness implies  the parabolic comparison principle  below.

\begin{proposition}\label{P:missing_key2}
Let $T>0$. Let $u,v: \mathbb{M}\times [0,T] \to \mathbb{R}$ be  smooth functions such that  for every $T>0$,  $\sup_{t \in [0,T]} \| u(\cdot,t)\|_\infty <\infty$, $\sup_{t \in [0,T]} \| v(\cdot,t)\|_\infty <\infty$; If the  inequality 
\[
\Dh u+\frac{\partial u}{\partial t} \ge v
\]
holds on $\mathbb{M}\times [0,T]$, then we have
\[
P_T(u(\cdot,T))(x) \ge u(x,0) +\int_0^T P_s(v(\cdot,s))(x) ds.
\]

\end{proposition}

\subsection{Li-Yau estimates}

We show in this section how to obtain the Li-Yau estimate which is a crucial ingredient to prove the Bonnet-Myers theorem.

Henceforth, we will indicate $C_b^\infty(\mathbb M) = C^\infty(\M)\cap L^\infty(\M)$. A key lemma is the following.

\begin{lemma}\label{L:derivatives}
Let $f \in C^\infty_b(\mathbb{M})$, $f > 0$ and $T>0$, and consider the functions
\[
\phi_1 (x,t)=(P_{T-t} f) (x)\Gamma (\ln P_{T-t}f)(x),
\]
\[
\phi_2 (x,t)= (P_{T-t} f)(x) \Gamma^{\mathcal{V}} (\ln P_{T-t}f)(x),
\]
which are defined on $\M\times [0,T)$.  We have
\[
\Dh \phi_1+\frac{\partial \phi_1}{\partial t} =2 (P_{T-t} f) \Gamma_2 (\ln P_{T-t}f). 
\]
and
\[
\Dh \phi_2+\frac{\partial \phi_2}{\partial t} =2 (P_{T-t} f) \Gamma_2^{\mathcal{V}} (\ln P_{T-t}f).
\]
\end{lemma}

\begin{proof}
This is direct computation without trick. Let us just point out that the formula
\[
\Dh \phi_2+\frac{\partial \phi_2}{\partial t} =2 (P_{T-t} f) \Gamma_2^{\mathcal{V}} (\ln P_{T-t}f).
\]
uses the fact that $\Gamma(g , \Gamma^\V (g))=\Gamma^\V (g , \Gamma (g))$ and thus that the foliation is totally geodesic.
\end{proof}

We now show how to prove the Li-Yau estimates for the horizontal semigroup. The method we use is adapted from \cite{BL2}

\begin{theorem}\label{T:ge}
Let $\alpha >2$. For  $f \in C_0^\infty(\M)$, $f  \ge 0$, $f \neq 0$, the following inequality holds for $t>0$:
\begin{align*}
 & \Gamma (\ln P_t f) +\frac{2 \rho_2}{\alpha}  t \Gamma^\mathcal{V} (\ln P_t f) \\
  \le & \left(1+\frac{\alpha \kappa}{(\alpha-1)\rho_2}-\frac{2\rho_1}{\alpha} t\right)
\frac{\Dh P_t f}{P_t f} +\frac{n\rho_1^2}{2\alpha} t-\frac{\rho_1 n}{2}\left(
1+\frac{\alpha \kappa}{(\alpha-1)\rho_2}\right) +\frac{n(\alpha-1)^2\left(
1+\frac{\alpha \kappa}{(\alpha-1)\rho_2}\right)^2}{8(\alpha-2)t}.
\end{align*}
\end{theorem}

\begin{proof}
We fix $T>0$ and consider two functions $a,b:[0,T] \to \mathbb{R}_{\ge 0}$ to be chosen later. Let $f \in C^\infty(\mathbb{M})$, $ f \ge 0$.
Consider the function
\[
\phi (x,t)=a(t)(P_{T-t} f) (x)\Gamma (\ln P_{T-t}f)(x)+b(t)(P_{T-t} f) (x) \Gamma^\mathcal{V} (\ln P_{T-t}f)(x).
\]
Applying Lemma \ref{L:derivatives} and the curvature-dimension inequality in Theorem  \ref{CD}, we obtain
\begin{align*}
 & \Dh \phi+\frac{\partial \phi}{\partial t} \\
=& a' (P_{T-t} f) \Gamma (\ln P_{T-t}f)+b' (P_{T-t} f) \Gamma^\mathcal{V} (\ln P_{T-t}f) +2a (P_{T-t} f) \Gamma_2 (\ln P_{T-t}f) \\
 & +2b (P_{T-t} f) \Gamma_2^\mathcal{V} (\ln P_{T-t}f) \\
\ge&  \left(a'+2\rho_1 a -2\kappa \frac{a^2}{b}\right)(P_{T-t} f) \Gamma (\ln P_{T-t}f)  +(b'+2\rho_2 a) (P_{T-t} f)  \Gamma^\mathcal{V} (\ln P_{T-t}f) \\
&+\frac{2a}{n}  (P_{T-t} f) (\Dh(\ln P_{T-t} f))^2. 
\end{align*}
But, for any function $\gamma:[0,T]\to \R$
\[
(\Dh(\ln P_{T-t} f))^2 \ge 2\gamma \Dh(\ln P_{T-t}f) -\gamma^2,
\]
and from chain rule
\[
 \Dh(\ln P_{T-t}f)=\frac{\Dh P_{T-t}f}{P_{T-t}f} -\Gamma(\ln P_{T-t} f ).
 \]
Therefore, we obtain
 \begin{align*}
\Dh \phi+\frac{\partial \phi}{\partial t}   \ge & \left(a'+2\rho_1 a -2\kappa \frac{a^2}{b}-\frac{4a\gamma}{n} \right) (P_{T-t} f) \Gamma (\ln P_{T-t}f)
\\
& +(b'+2\rho_2 a)  (P_{T-t} f) \Gamma^\mathcal{V} (\ln P_{T-t}f) +\frac{4a\gamma}{n} \Dh P_{T-t} f - \frac{2a\gamma^2}{n} P_{T-t} f.
\end{align*}
The idea is now to chose $a,b,\gamma$ such that
\begin{align*}
\begin{cases}
a'+2\rho_1 a -2\kappa \frac{a^2}{b}-\frac{4a\gamma}{n} =0 \\
b'+2\rho_2 a=0
\end{cases}
\end{align*}
With this choice we get
\begin{align}\label{estimetr}
\Dh \phi+\frac{\partial \phi}{\partial t}   \ge \frac{4a\gamma}{n} \Dh P_{T-t} f - \frac{2a\gamma^2}{n} P_{T-t} f
\end{align}

We wish to apply Proposition \ref{P:missing_key2}. So, we take $f \in C_0^\infty(\M)$ and apply the previous inequality with $f_\varepsilon=f+\varepsilon$ instead of $f$, where $\varepsilon >0$. If moreover $a(T)=b(T)=0$, we end up with the inequality
\begin{align}\label{jkli}
 & a(0)(P_{T} f_\varepsilon) (x)\Gamma (\ln P_{T}f_\varepsilon)(x)+b(0)(P_{T} f) (x) \Gamma^\mathcal{V} (\ln P_{T}f_\varepsilon)(x) \notag \\
 \le &  -\int_0^T \frac{4a\gamma}{n} dt \Dh P_{T} f_\varepsilon (x)  +\int_0^T \frac{2a\gamma^2}{n}dt  P_{T} f_\varepsilon(x)
\end{align}
If we now chose $b(t)=(T-t)^\alpha$ and $b,\gamma$ such that
\begin{align*}
\begin{cases}
a'+2\rho_1 a -2\kappa \frac{a^2}{b}-\frac{4a\gamma}{n} =0 \\
b'+2\rho_2 a=0
\end{cases}
\end{align*}
the result follows by a simple computation and sending then $\varepsilon \to 0$.
\end{proof}

Observe that if $\mathfrak{Ric}_{\mathcal{H}} \ge 0$, then we can take $\rho_1 =0$ and the estimate simplifies to
\begin{align*}
 \Gamma (\ln P_t f) +\frac{2 \rho_2}{\alpha}  t \Gamma^\mathcal{V} (\ln P_t f)  \le & \left(1+\frac{\alpha \kappa}{(\alpha-1)\rho_2}\right) \frac{\Dh P_t f}{P_t f}+\frac{n(\alpha-1)^2\left(
1+\frac{\alpha \kappa}{(\alpha-1)\rho_2}\right)^2}{8(\alpha-2)t}.
\end{align*}

By adapting a classical method of Li and Yau \cite{LY} and integrating this last inequality on sub-Riemannian geodesics leads to a parabolic Harnack inequality  (details are in \cite{BG}). For $\alpha >2$, we denote
\begin{align}\label{D}
D_\alpha=\frac{n(\alpha-1)^2\left(
1+\frac{\alpha \kappa}{(\alpha-1)\rho_2}\right)}{4(\alpha-2)}.
\end{align}
The minimal value of $D_\alpha$ is difficult to compute, depends on $\kappa, \rho_2$ and does not seem relevant because the constants we get are anyhow not optimal. We just point out that the choice $\alpha=3$ turns out to simplify many computations and is actually optimal when $\kappa=4\rho_2$.

\begin{corollary}
Let us assume that $\mathfrak{Ric}_{\mathcal{H}} \ge 0$. Let $f\in
L^\infty(\bM)$, $f \ge 0$, and consider $u(x,t) =
P_t f(x)$. For every $(x,s), (y,t)\in \bM\times (0,\infty)$ with
$s<t$ one has with $D_\alpha$ as in \eqref{D}
\begin{equation*}
u(x,s) \le u(y,t) \left(\frac{t}{s}\right)^{\frac{D_\alpha}{2}} \exp\left(
\frac{D_\alpha}{n} \frac{d(x,y)^2}{4(t-s)} \right).
\end{equation*}
Here $d(x,y)$ is the sub-Riemannian distance between $x$ and $y$.
\end{corollary}

It is classical since the work by Li and Yau (see \cite{LY}) and not difficult  to prove that a parabolic Harnack inequality implies a Gaussian upper bound on the heat kernel. With the curvature dimension inequality in hand, it is actually also possible, but much more difficult, to prove a lower bound. The final result proved in \cite{BBG} is:

\begin{theorem} Let us assume that $\mathfrak{Ric}_{\mathcal{H}} \ge 0$, then for any $0<\ve <1$
there exists a constant $C(\ve) = C(n,\kappa,\rho_2,\ve)>0$, which tends
to $\infty$ as $\ve \to 0^+$, such that for every $x,y\in \bM$
and $t>0$ one has
\[
\frac{C(\ve)^{-1}}{\mu(B(x,\sqrt
t))} \exp
\left(-\frac{D_\alpha d(x,y)^2}{n(4-\ve)t}\right)\le p_t(x,y)\le \frac{C(\ve)}{\mu(B(x,\sqrt
t))} \exp
\left(-\frac{d(x,y)^2}{(4+\ve)t}\right),
\]
where $p_t(x,y)$ is the heat kernel of $\Dh$.
\end{theorem}

We mention that those results are not restricted to the case $\rho_1=0$ but that similar results may also  be obtained when $\rho_1 \le 0$. We refer to   \cite{BBGM}.

\section{The horizontal Bonnet-Myers theorem}

Let $\M$ be a smooth, connected  manifold with dimension $n+m$. We assume that $\bM$ is equipped with a Riemannian foliation $\mathcal{F}$ with bundle like complete metric $g$ and totally geodesic  $m$-dimensional leaves. We also assume that the horizontal distribution is of Yang-Mills type.

\

In this section, we prove the following result:

\begin{theorem}\label{bonnet}
Assume that for any smooth horizontal one-form $\eta \in \Gamma^\infty(\mathcal{H}^*)$,
\[
\langle \mathfrak{Ric}_{\mathcal{H}}(\eta),\eta \rangle_\mathcal{H} \ge \rho_1 \| \eta \|^2_\mathcal{H},\quad  \left\langle -\mathbf{J}^2(\eta), \eta \right\rangle_\mathcal{H} \le \kappa \| \eta \|^2_\mathcal{H},
\]
and that for any vertical one-form  $\eta \in \Gamma^\infty(\mathcal{V^*})$,
\[
\frac{1}{4} \mathbf{Tr} ( J_\eta^* J_\eta) \ge \rho_2 \| \eta \|_\mathcal{V}^2,
\]
with $\rho_1,\rho_2 >0$ and $\kappa \ge 0$. Then the manifold $\M$ is compact  and we have
 \[  \mathbf{diam}\ \bM \le 2\sqrt{3} \pi \sqrt{
\frac{\rho_2+\kappa}{\rho_1\rho_2} \left(
1+\frac{3\kappa}{2\rho_2}\right)n } ,
\]
where  $\mathbf{diam}\ \bM$ is the diameter of $\M$ for the sub-Riemannian distance.

\end{theorem}

We mention that the bound 
\begin{align*} \mathbf{diam}\ \bM \le 2\sqrt{3} \pi \sqrt{
\frac{\kappa+\rho_2}{\rho_1\rho_2} \left(
1+\frac{3\kappa}{2\rho_2}\right)n }.
\end{align*}
is not sharp. This is because the method we use, that comes from the joint work with Garofalo \cite{BG}  is an adaptation of the energy-entropy inequality methods developped by Bakry in \cite{bakry-stflour}. Even in the Riemannian case, Bakry's methods are known to lead to non sharp constants.  An analytical method that leads to sharp diameter constants is based on sharp Sobolev inequalities (see \cite{BL}), however as of today, this is still an open question to prove those sharp Sobolev inequalities.

\subsection{Ultracontractivity bounds and diameter estimates}

In this section, we show how ultracontractivity bounds for the heat semigroup can be used to get diameter bounds on a space. The result we give below is a variation on results due to Bakry \cite{bakry-stflour} and Davies \cite{Davies}. The result holds true in a great generality in the context of Dirichlet spaces.

\ 

Let $\mu$ be a probability measure on a locally compact topological space $\Omega$. We assume that there is on $\Omega$ a regular Dirichlet form $\mathcal{E}$ which is symmetric in $L^2 (\Omega,\mu)$ (see Fukushima \cite{Fu}) . Let $(P_t)_{t \ge 0}$ be the symmetric Markov semigroup associated to $\mathcal{E}$. It is well-known that we can associated to $\mathcal{E}$ a distance which is defined as follows.

\

Let $\mathcal{D}$ be the domain in  $L^2 (\Omega,\mu)$ of the Dirichlet form $\mathcal{E}$. We denote by $\mathcal{D}_\infty$ the set of bounded functions in $\mathcal{D}$. For $f\in \mathcal{D}_\infty$, we define 
\[
I_f(h)= \frac{1}{2} (2\mathcal{E}(fg,g) )-\mathcal{E}(f^2,g)), \quad g \in \mathcal{D}_\infty.
\]
We then say that $f \in \mathbf{Lip}$ if for every $g \in \mathcal{D}_\infty $,
\[
| I_f(g) | \le \| g \|_1.
\]
For $x,y \in \Omega$, we define
\[
d(x,y)=\sup \{ f(x)-f(y), f \in \mathbf{Lip} \}
\]
and assume that $d$ is a distance everywhere finite that induces the topology of $\Omega$.

\begin{theorem}\label{ultracontractivity_compact}
Assume that for every $f \in L^2 (\Omega,\mu)$ and $t \ge 0$,
\[
\|P_t f \|_\infty \le \frac{1}{\left( 1-e^{-\alpha t}
\right)^{\frac{D}{2} }} \| f \|_2,
\]
with $\alpha,D >0$. Then $\Omega$ is compact and its diameter for the distance $d$ satisfies
\[
\mathbf{diam} (\Omega) \le 2\pi \sqrt{\frac{2D}{\alpha}}.
\] 
\end{theorem}

\begin{proof}
Since
\[
\|P_t f \|_\infty \le \frac{1}{\left( 1-e^{-\alpha t}
\right)^{\frac{D}{2} }} \| f \|_2,
\]
from Davies' theorem (Theorem 2.2.3 in \cite{Davies}), for $f \in L^2
(\Omega)$ such that $\int_\Omega f^2 d\mu =1$, we obtain
\[
\int_\Omega f^2 \ln f^2 d\mu \le 2t \int_\Omega \Gamma(f) d\mu -D \ln \left( 1-e^{-\alpha t} \right), \quad t >0.
\]
By minimizing over $t$ the right-hand side of the above inequality, we get that for $f \in L^2 (\bM)$ such that $\int_\bM f^2 d\mu =1$
\[
\int_\bM f^2 \ln f^2 d\mu \le \Phi \left( \int_\bM \Gamma(f) d\mu \right),
\]
where
\[
\Phi(x)=  D \left[  \left( 1+\frac{2}{\alpha D } x\right)\ln \left(
1+\frac{2}{\alpha D} x\right)-\frac{2}{\alpha D } x  \ln \left(
\frac{2}{\alpha D} x  \right)\right].
\]

The function $\Phi$ enjoys the following properties:
\begin{itemize}
\item $\Phi'(x)/x^{1/2}$  and $\Phi(x)/x^{3/2}$ are integrable on $(0,\infty)$;
\item $\Phi$ is concave;
\item $\frac{1}{2}\int_0^{+\infty} \frac{\Phi(x)}{x^{3/2}}dx=\int_0^{+\infty} \frac{\Phi'(x)}{\sqrt{x}}dx =-2\int_0^{+\infty} \sqrt{x} \Phi''(x)dx <+\infty.$
\end{itemize}
We can therefore apply Theorem 5.4 in \cite{bakry-stflour} to deduce that the diameter of $\bM$ is finite and
\[
\mathbf{diam} (\Omega) \le-2\int_0^{+\infty} \sqrt{x} \Phi''(x)dx.
\]
Since $\Phi''(x) = - \frac{2D}{x(2x+\alpha D)}$, a routine
calculation shows
\[
-2\int_0^{+\infty} \sqrt{x} \Phi''(x)dx=2\pi \sqrt{\frac{2D}{\alpha}}.
\]
\end{proof}

\subsection{Proof of the compactness theorem}

We now turn to the proof of Theorem \ref{bonnet}. Let $\M$ be a smooth, connected  manifold with dimension $n+m$. As usual, we assume that $\bM$ is equipped with a Riemannian foliation with bundle like complete metric $g$ and totally geodesic  $m$-dimensional leaves. We also assume that the horizontal distribution is of Yang-Mills type and  that for any smooth horizontal one-form $\eta \in \Gamma^\infty(\mathcal{H}^*)$,
\[
\langle \mathfrak{Ric}_{\mathcal{H}}(\eta),\eta \rangle_\mathcal{H} \ge \rho_1 \| \eta \|^2_\mathcal{H},\quad  \left\langle -\mathbf{J}^2(\eta), \eta \right\rangle_\mathcal{H} \le \kappa \| \eta \|^2_\mathcal{H},
\]
and that for any vertical  $\eta \in \Gamma^\infty(\mathcal{V^*})$,
\[
\frac{1}{4} \mathbf{Tr} ( J_\eta^* J_\eta) \ge \rho_2 \| \eta \|_\mathcal{V}^2,
\]
with $\rho_1,\rho_2 >0$ and $\kappa \ge 0$. 

\

The first step is to prove that the volume of $\M$ is finite.

\begin{lemma}
The measure $\mu$ is finite, i.e.
$\mu(\bM) <+\infty$
and for every $x \in \bM$, $f \in L^2(\M)$,
\[
P_t f (x) \to_{t \to +\infty} \frac{1}{\mu(\bM)} \int_\bM f d\mu.
\]
\end{lemma}

\begin{proof}
We first prove a gradient bound for the semigroup following an argument close to the one in the proof of Theorem \ref{T:ge}.

We fix $T>0$ and consider two functions $a,b:[0,T] \to \mathbb{R}_{\ge 0}$ to be chosen later. Let $f \in C_0^\infty(\mathbb{M})$.
Consider the function
\[
\phi (x,t)=a(t)\Gamma (P_{T-t}f)(x)+b(t) \Gamma^\mathcal{V} ( P_{T-t}f)(x).
\]
Computing derivatives and applying  the curvature-dimension inequality in Theorem  \ref{CD}, we obtain
\begin{align*}
 & \Dh \phi+\frac{\partial \phi}{\partial t} \\
=& a' \Gamma ( P_{T-t}f)+b'  \Gamma^\mathcal{V} ( P_{T-t}f) +2a  \Gamma_2 (P_{T-t}f)  +2b  \Gamma_2^\mathcal{V} ( P_{T-t}f) \\
\ge&  \left(a'+2\rho_1 a -2\kappa \frac{a^2}{b}\right) \Gamma ( P_{T-t}f)  +(b'+2\rho_2 a)   \Gamma^\mathcal{V} ( P_{T-t}f).
\end{align*}
Let us now chose
\[
b(t)=e^{-\frac{2\rho_1 \rho_2 t}{\kappa+\rho_2}}
\]
and
\[
a(t)=-\frac{b'(t)}{2\rho_2},
\]
so that
\begin{align*}
b'+2\rho_2 a=0
\end{align*}
and
\begin{align*}
a'+2\rho_1 a -2\kappa \frac{a^2}{b}  =0.
\end{align*}
With this choice, we get
\[
 \Dh \phi+\frac{\partial \phi}{\partial t}  \ge 0.
\]
From the parabolic comparison theorem Theorem \ref{P:missing_key2}, we deduce then 
 \[
  \Gamma( P_t f) +\frac{\kappa+\rho_2}{\rho_1} \Gamma^\mathcal{V} ( P_t f) \le e^{-2\frac{\rho_1 \rho_2 t}{\kappa+\rho_2}} \left( P_t (  \Gamma( f)) +\frac{\kappa+\rho_2}{\rho_1}  P_t ( \Gamma^\mathcal{V}( f))\right)
 \]

\

 Let $f,g \in  C^\infty_0(\mathbb M)$,  we have
\begin{align*}
\int_{\bM} (P_t f -f) g d\mu & = \int_0^t \int_{\bM}\left(
\frac{\partial}{\partial s} P_s f \right) g d\mu ds \\
 & = \int_0^t
\int_{\bM}\left(\Dh P_s f \right) g d\mu ds \\
 & =- \int_0^t \int_{\bM} \Gamma ( P_s f , g) d\mu ds.
\end{align*}
By means of the previous bound and Cauchy-Schwarz inequality we
find that
\begin{equation}\label{P1}
\left| \int_{\bM} (P_t f -f) g d\mu \right| \le \left(\int_0^t
e^{-\frac{\rho_1 \rho_2 s}{\kappa+\rho_2}} ds\right) \sqrt{ \| \Gamma (f) \|_\infty +\frac{\kappa+\rho_2}{\rho_1} \|
\Gamma^\mathcal{V} (f) \|_\infty } \int_{\bM}\Gamma (g)^{\frac{1}{2}}d\mu.
\end{equation}

\

It is seen from the spectral theorem that in $L^2(\bM,\mu)$ we have  a convergence $P_t f \to P_\infty f$, where $P_\infty f$ belongs to the domain of $\Dh$. Moreover $\Dh P_\infty f=0$. By hypoellipticity of $\Dh$ we deduce that $P_\infty f$ is a smooth function. Since $\Dh P_\infty f=0$, we have $\Gamma(P_\infty f)=0$ and therefore $P_\infty f $ is constant.

\

 Let us now assume that $\mu(\bM)=+\infty$. This implies in particular that $P_\infty f =0$  because no constant besides $0$ is in $L^2(\bM,\mu)$. Using then (\ref{P1}) and letting $t \to +\infty$, we infer
 \begin{equation*}
\left| \int_{\bM} f g d\mu \right| \le  \left(\int_0^{+\infty}
e^{-\frac{\rho_1 \rho_2 s}{\kappa+\rho_2}} ds\right) \sqrt{ \| \Gamma (f) \|_\infty +\frac{\kappa+\rho_2}{\rho_1} \|
\Gamma^\mathcal{V} (f) \|_\infty } \int_{\bM}\Gamma (g)^{\frac{1}{2}}d\mu
\end{equation*}
Let us assume $g \ge 0$, $g \ne 0$ and take for $f$ a sequence $h_n$   increasing in $ C_0^\infty(\M)$, $0 \le h_n \le 1$,  such that $h_n\nearrow 1$ on $\mathbb{M}$, and $||\Gamma (h_n)||_{\infty} \to 0$  Letting $n \to \infty$, we deduce
\[
\int_\bM g d\mu \le 0,
\]
which is clearly absurd. As a consequence $\mu (\bM) <+\infty$. 

\

The invariance of $\mu$ by the semigroup implies
\[
\int_\bM P_\infty f d\mu =\int_\bM f d\mu,
\]
and thus
\[
P_\infty f =\frac{1}{\mu(\bM)} \int_\bM f d\mu.
\]
Finally, using the Cauchy-Schwarz inequality, we find that for $x \in \bM$, $f \in L^2(\bM, \mu)$, $s,t,\tau \ge 0$,
\begin{align*}
| P_{t+\tau} f (x)-P_{s+\tau} f (x) | & = | P_\tau (P_t f -P_s f) (x) | \\
 &=\left| \int_\bM p(\tau, x, y) (P_t f -P_s f) (y) \mu(dy) \right| \\
 &\le \int_\bM p(\tau, x, y)^2 \mu(dy) \| P_t f -P_s f\|^2_2 \\
 &\le p(2\tau,x,x) \| P_t f -P_s f\|^2_2.
\end{align*}
Thus, we have 
\[
P_t f (x) \to_{t \to +\infty} \frac{1}{\mu(\bM)} \int_\bM f d\mu.
\]
\end{proof}

Since $\mu (\M) <+\infty$, we can assume $\mu (\M)=1$. The second lemma we need is a uniform bound on the heat kernel of $\Dh$, from which we will immediately deduce Theorem \ref{bonnet} by using  Theorem \ref{ultracontractivity_compact}.

\begin{lemma}
Let $\beta >2$. For $f \in C_0^\infty(\M)$, $f \ge 0$, and $t \ge 0$,
\[
P_tf \le \frac{1}{\left(1-e^{- \frac{2\rho_1 \rho_2 t}{\beta (\rho_2+\kappa)} }  \right)^{D_\beta/2}}\int_\M f d \mu,
\]
where
\[
D_\beta=\frac{n}{4} \frac{\beta-1}{\beta -2} \left( \left( 1+\frac{\kappa}{\rho_2}\right) \beta -1 \right) .
\]
\end{lemma}

\begin{proof}
We fix $T>0$ and consider  functions $a,b:[0,T] \to \mathbb{R}_{\ge 0}$ and $\gamma:[0,T]\to \R$ such that
\begin{align*}
\begin{cases}
a'+2\rho_1 a -2\kappa \frac{a^2}{b}-\frac{4a\gamma}{n} =0 \\
b'+2\rho_2 a=0 \\
a(T)=b(T)=0
\end{cases}
\end{align*}

Let $f \in C_0^\infty(\M)$, $f \ge 0$. Recall that from the inequality \eqref{jkli}
\begin{align*}
 & a(0)(P_{T} f) (x)\Gamma (\ln P_{T}f)(x)+b(0)(P_{T} f) (x) \Gamma^\mathcal{V} (\ln P_{T}f)(x) \notag \\
 \le &  -\int_0^T \frac{4a\gamma}{n} dt \Dh P_{T} f (x)  +\int_0^T \frac{2a\gamma^2}{n}dt  P_{T} f(x).
\end{align*}
Snce the left-hand side is non negative we deduce
\[
-2\int_0^T a\gamma dt \Dh P_{T} f (x)  +\int_0^T a\gamma^2 dt  P_{T} f(x) \ge 0.
\]
Let us choose
\[
b(t)=(e^{-\alpha t} -e^{-\alpha T})^\beta,\ \ \ \ 0\le t\le T,
\]
with $\beta >2$, $\alpha= \frac{2 \rho_1\rho_2}{\beta(\rho_2+\kappa)}$ and $a,\gamma$ such that
\begin{align*}
\begin{cases}
a'+2\rho_1 a -2\kappa \frac{a^2}{b}-\frac{4a\gamma}{n} =0 \\
b'+2\rho_2 a=0
\end{cases}
\end{align*}

We obtain after some computations:
\begin{align}\label{gamma_bound}
0 \le 
\frac{(2\rho_1-\alpha)}{2\rho_2 \left(1-\frac{1}{\beta} \right)}
e^{-\alpha T} \frac{\Dh P_T f}{P_T f}+ \frac{n(2\rho_1-\alpha)^2}{16 \rho_2 \left(1-\frac{2}{\beta}
\right)} \frac{e^{-2\alpha T}}{ 1-e^{-\alpha T}}. \notag
\end{align}
Since $T$ is arbitrary, this  implies that for every $t>0$,
\[
\frac{\Dh P_t f}{P_t f} \ge -\frac{n}{8} \frac{\beta -1}{\beta -2} (2 \rho_1 -\alpha) \frac{e^{-\alpha t}}{1-e^{-\alpha t}}
\]
Taking into account that $P_t f (x) \to_{t \to +\infty}  \int_\bM f d\mu$ and integrating from $0$ to $+\infty$ yields the claim.
\end{proof}
As a consequence of Theorem \ref{ultracontractivity_compact}, we deduce that $\M$ is compact and that for every $\beta >2$,
\[
\mathbf{diam} (\M)\le \pi \left( 1+\frac{\kappa}{\rho_2}\right) \sqrt{\frac{n}{\rho_1}} \sqrt{ \frac{\beta(\beta -1)}{\beta-2} \left( \beta - \frac{\rho_2}{\rho_2+\kappa} \right) }.
\]
The optimal $\beta$ does not lead to a nice formula. The value $\beta=3$ yields the bound
 \[ \mathbf{diam}\ \bM \le 2\sqrt{3} \pi \sqrt{
\frac{\rho_2+\kappa}{\rho_1\rho_2} \left(
1+\frac{3\kappa}{2\rho_2}\right)n }.
\]

\section{Riemannian foliations and hypocoercivity}

We now show how the geometry of foliations can be used to study some hypoelliptic diffusion operators that we call  Kolmogorov type operators. We shall mainly be interested in the problem of convergence to equilibrium for the parabolic equation associated with those operators. The methods we develop to prove convergence to equilibrium come with estimates that Villani call hypocoercive (see \cite{Villani1}). As an illustration, we study the kinetic Fokker-Planck equation.

\subsection{Kolmogorov type operators}

Let $\M$ be a smooth, connected  manifold with dimension $n+m$. We assume that $\bM$ is equipped with a Riemannian foliation $\mathcal{F}$ with  $m$-dimensional leaves. As before, we indicate by $\Dh$ the horizontal Laplacian and by $\Dv$ the vertical Laplacian.

\begin{definition}
We call Kolmogorov type operator, a hypoelliptic diffusion operator $L$ on $\M$ that can be written as
\[
L=\Dv+Y,
\]
where $Y$ is a smooth vector field on $\M$.
\end{definition}

The simplest example of such an operator was studied by Kolmogorov himself. Let us consider the following operator
\[
L=\frac{\partial^2}{\partial v^2}+v \frac{\partial}{\partial x}.
\]
Then, by considering the trivial foliation on $\R^2$ that comes from the submersion $(v,x) \to x$, we can write $L=\Dv+Y$ with $\Dv=\frac{\partial^2}{\partial v^2}$ and $Y=v \frac{\partial}{\partial x}$. 

\

More interesting is the operator on $\R^{2n}=\{(v,x), v \in \R^n,  x \in \R^n \}$,
\[
L=\Delta_v  - v \cdot \nabla_v +\nabla_x V \cdot \nabla_v -v\cdot \nabla_x,
\]
where $V: \R^n \to \R$ is a smooth potential. The parabolic equation 
\begin{equation}\label{FP1}
\frac{\partial h}{\partial t}=\Delta_v h - v \cdot \nabla_v h+\nabla V \cdot \nabla_v h -v\cdot \nabla_x h , \quad (x,v) \in \mathbb{R}^{2n}.
\end{equation}
is then known as the kinetic Fokker-Planck equation with confinement potential $V$. It has been extensively studied due its importance in mathematical physics. We refer for instance to \cite{EH,HN1,Villani1,Wu}.
This equation is the Kolmogorov-Fokker-Planck equation associated to the stochastic differential system
\[
\begin{cases}
dx_t =v_t dt \\
dv_t=-v_t dt -\nabla V (x_t) dt +dB_t,
\end{cases}
\]
where $(B_t)_{ t\ge 0}$ is a Brownian motion in $\mathbb{R}^n$. 

We can obviously write
\[
L=\Dv+Y,
\]
where $\Dv=\Delta_v$ and $Y=- v \cdot \nabla_v +\nabla_x V \cdot \nabla_v -v\cdot \nabla_x$, and consider the trivial foliation on $\R^{2n}$ that comes from the submersion $(v,x) \to x$. However we will see that the metric on $\R^{2n}$ to chose is not the standard Euclidean metric but rather the metric that makes
\[
\left\{ 2\frac{\partial}{\partial x_i}+\frac{\partial }{\partial v_i}, \frac{\partial }{\partial v_i}, 1 \le i \le n \right\},
\]
an orthonormal basis at any point.

\subsection{Convergence to equilibrium and hypocoercive estimates}

We consider a Kolmogorov type operator 
\[
L=\Dv+Y,
\]
and assume in this section that the Riemannian foliation is totally geodesic with a bundle like metric. Our first task will be to prove a Bochner's type inequality for $L$. If $f \in C^\infty(\M)$, we denote
\[
\mathcal{T}_2(f)=\frac{1}{2} \left( L (\| \nabla f \|^2) -2\langle \nabla f , \nabla L f \rangle \right),
\]
where $\nabla$ is the whole Riemannian gradient.  We denote by $\mathbf{Ric}_\mathcal{V}$ the  Ricci curvature of the leaves and we denote by $DY$ the tensor defined by $DY(U,V)=\langle D_U Y, V \rangle$ where $D$ is the Levi-Civita connection. 

\

We have then the following Bochner's inequality,
\begin{theorem}
For every $f \in C^\infty(\M)$,
\[
\mathcal{T}_2(f) \ge (\mathbf{Ric}_\mathcal{V}  -DY ) (\nabla f, \nabla f).
\]
\end{theorem}

\begin{proof}
We can split $\mathcal{T}_2$ into three parts:
\[
\mathcal{T}_2(f)=\Gamma_2^\Ho(f) +\Gamma_2^\V (f)+\Gamma_2^Y (f),
\]
where
\[
\Gamma_2^\Ho(f) =\frac{1}{2} \left( \Dv (\| \nabla_\Ho f \|^2) -2\langle \nabla_\Ho f , \nabla_\Ho \Dv f \rangle \right),
\]
\[
\Gamma_2^\V(f) =\frac{1}{2} \left( \Dv (\| \nabla_\V f \|^2) -2\langle \nabla_\V f , \nabla_\V \Dv f \rangle \right),
\]
and
\[
\Gamma_2^Y (f)=\frac{1}{2} \left( Y (\| \nabla f \|^2) -2\langle \nabla f , \nabla Y f \rangle \right).
\]
We now compute these three terms separately.  

\

Since $\nabla_\Ho$ and $\Dv$ commute, we find:
\[
\Gamma_2^\Ho(f) =\| \nabla_\V \nabla_\Ho f \|^2.
\]
So we have, $\Gamma_2^\Ho(f) \ge 0$. 

\

Since $\Dv$ is the Laplace-Beltrami operator on the leaves, from the usual Bochner's formula we have:
\[
\Gamma_2^\V(f) =\| \nabla^2_\V f \|^2 +\mathbf{Ric}_\mathcal{V} (\nabla f , \nabla f).
\]
Thus we have
\[
\Gamma_2^\V(f) \ge \mathbf{Ric}_\mathcal{V} (\nabla f , \nabla f).
\]

\

Finally, we see that
\[
\frac{1}{2} Y \| \nabla f \|^2 =\frac{1}{2}  D_Y  \| \nabla f \|^2=\langle \nabla f , D_Y \nabla f \rangle
\]
and
\[
\langle \nabla f , \nabla Y f \rangle =\langle \nabla f , \nabla \langle Y, \nabla f \rangle  \rangle=DY(\nabla f, \nabla f) + \langle \nabla f , D_Y \nabla f \rangle.
\]
\end{proof}

A difficulty that arises when studying Kolmogorov type operators is that, in general, they are not symmetric with respect to any measure. As a consequence, we can not use functional analysis and the spectral theory of self-adjoint operators to define the semigroup generated $L$. A typical assumption to ensure that $L$ generates a well-behaved semigroup is the existence of a nice Lyapounov function. 

\

 So, in the sequel,  we will  assume that there exists a function $W$ such that $W \ge 1$, $\| \nabla W \| \le C W$, $LW \le C W$ for some constant $C>0$ and $\{ W \le m \}$ is compact for every $m$.  This condition is actually not too restrictive and may be checked in concrete situations. If $\M$ is compact, it is obviously satisfied. A non-compact situation where it is satisfied is the following: Assume that $\M$ is non-compact and that any two points of $\M$ can be joined by a unique geodesic. Also assume that the Riemannian foliation comes from a Riemannian submersion $\pi: \M \to \mathbb{B}$ and that the Ricci curvature of the leaves $\mathbf{Ric}_\mathcal{V}$ is bounded from below by a negative constant $-K$. If $x \in \B$, we denote $\mathcal{L}_x=\pi^{-1} ( \{ x \})$. Any geodesic $\gamma:[0,L] \to \mathbb{B}$ in the base space can be lifted into a geodesic in $\M$.  For $x \in \mathcal{L}_{\gamma(0)}$,   denote $\tau_\gamma (x)$ the endpoint of of the unique horizontal lift of $\gamma$ starting from $x$. Since the leaves are assumed to be totally geodesic, the map $\tau_\gamma$ induces an isometry between $\mathcal{L}_{\gamma(0)}$ and $\mathcal{L}_{\gamma(L)}$. Fix now a base point $x_0 \in \M$ and for $x \in \mathcal{L}_{\pi(x_0)}$ denote $\rho_\V (x)=d(x_0,x)$. If $ x \notin \mathcal{L}_{\pi(x_0)}$, then consider $\gamma:[0,L] \to \B$ to be the unique geodesic between $\pi(x_0)$ and $\pi(x)$ and define $\rho_\V(x)=d(\tau_\gamma(x_0), x)$. Consider also the function $\rho_\Ho (x)=d (\pi (x_0), \pi ( x))$ and finally define
 \[
 W(x)=1+\rho_\V(x)^2+\rho_\Ho (x)^2.
 \]
Obviously $W \ge 1$, is smooth, and such that $\{ W \le m \}$ is compact for every $m$. We have,
\[
\| \nabla W \| = 2 \rho_\V \| \nabla \rho_V \|	 +2 \rho_\Ho \| \nabla \rho_\Ho \|\le 	 2 \rho_\V 	 +2 \rho_\Ho \le C W,
\]
and 
\[
LW=2 \rho_\V \Dv \rho_\V+2\rho_\Ho \Dv \rho_\Ho +2 \| \nabla_\V \rho_\V \|^2 +2 \| \nabla_\V \rho_\Ho \|^2+YW.
\]
On the other hand from the Laplacian comparison theorem on the leaves, $$\Dv \rho_\V \le (m-1)\sqrt{\frac{K}{m-1} }\coth \left(\sqrt{\frac{K}{m-1}} \rho_\V \right),$$ so we have for some constant $C$,
\[
LW \le CW+YW.
\]
So, if we additionally assume that $Y$ is a Lipschitz vector field, that is, $\| DY \| \le C$, then $W$ satisfies all the requirements.

\

The assumption about the existence of the function $W$ such that $LW \le CW$ easily implies that $L$ is the generator of a Markov semigroup $(P_t)_{t \ge 0}$ that uniquely solves the heat equation in $L^\infty$. Moreover, consider a smooth and decreasing function $h: \R_{\ge 0} \to \R$ such that $h=1$ on $[0,1]$ and $h=0$ on $[2,+\infty)$. Denote then $h_n=h \left( \frac{W}{n} \right)$ and consider the compactly supported diffusion operator
\[
L_n=h^2_n L.
\]
Since $L_n$ is compactly supported, a Markov semigroup $P_t^n$ with generator  $L_n$ is easily constructed as the unique bounded  solution of  $\frac{\partial P_t^n f}{\partial t} =L_nP^n_t f$, $f \in L^\infty$. Then, for every bounded $f$,
\[
P_t^n f \to P_t f , \quad n \to \infty.
\]

\

We now prove our first gradient bound for the Kolmogorov type operator.

\begin{theorem}\label{GB18}
Let us assume that for some $K \in \R$,
\[
\mathbf{Ric}_\mathcal{V}  -DY \ge -K,
\]
then for every bounded and Lipchitz function $f \in C^\infty(\M)$, we have for $t \ge 0$
\[
\| \nabla P_t f \|^2 \le e^{2K t} P_t (\| \nabla f \|^2).
\]
\end{theorem}

\begin{proof}
We follow an approach  by F.Y. Wang \cite{FYW2}. We fix $t >0$, $n \ge 1$  and $f \in C^\infty(\M)$ compactly supported inside the set $\{ W \le n \}$. Consider the functional defined for $s \in [0,t]$ and evaluated at a fixed point $x_0$ in the set $\{ W \le n \}$:
\[
\Phi_n(s) =P^n_s ( \| \nabla P^n_{t-s} f \|^2).
\]
We have
\[
\Phi'_n(s) =P^n_s (L_n  \| \nabla P^n_{t-s} f \|^2-2 \langle  \nabla L_n P^n_{t-s} f,  \nabla P^n_{t-s} f\rangle ).
\]
Now, observe that by assumption, and denoting $K^{-}$ the negative part of $K$, 
\begin{align*}
 & L_n  \| \nabla P^n_{t-s} f \|^2-2 \langle  \nabla L_n P^n_{t-s} f,  \nabla P^n_{t-s} f\rangle \\
   =& h_n^2 \mathcal{T}_2 (  P^n_{t-s} f, P^n_{t-s} f) -4h_nL P^n_{t-s} f  \langle  \nabla h_n  ,  \nabla P^n_{t-s} f\rangle \\
 \ge & -2K h_n^2 \| \nabla P^n_{t-s} f \|^2 -4h_nL P^n_{t-s} f  \langle  \nabla h_n  ,  \nabla P^n_{t-s} f\rangle \\
  \ge &  -2K h_n^2 \| \nabla P^n_{t-s} f \|^2 -4 P^n_{t-s} L_n f  \langle  \nabla \ln h_n  ,  \nabla P^n_{t-s} f\rangle \\
  \ge & -2K h_n^2 \| \nabla P^n_{t-s} f \|^2-4 \| Lf \|_\infty \| \nabla \ln h_n \| \| \nabla P^n_{t-s} f \| \\
   \ge &   -(2K^{-}+2)   \| \nabla P^n_{t-s} f \|^2-2 \| Lf \|^2_\infty \| \nabla \ln h_n \|^2 . 
\end{align*}
The term $\| \nabla \ln h_n \|$ can be estimated as follows inside the set $\{ W \le 2n \}$
\[
\| \nabla \ln h_n \| =- \frac{1}{n h_n } h' \left( \frac{W}{n} \right) \| \nabla W \| \le \frac{C}{ h_n },
\]
where $C$ is a constant independent from $n$. On the other hand a direct computation and the assumptions on $W$ show that
\[
L_n \left( \frac{1}{h_n^2} \right) \le \frac{C}{h^2_n},
\]
where, again,  $C$ is a constant independent from $n$. This last estimate classically implies
\[
P^n_s \left( \frac{1}{h_n^2} \right) \le  \frac{e^{Cs}}{h^2_n}.
\]
Putting the pieces together we end up with a differential inequality
\[
\Phi'_n(s) \ge   -(2K^{-}+2) \Phi_n(s)-C,
\]
where $C$ now depends on $f$ and  $t$, but still does not depend on $n$. Integrating this inequality from $0$ to $t$, yields a bound of the type
\[
\| \nabla P^n_t f \| \le C,
\]
where  $C$ depends on $f$ and  $t$. This bounds holds uniformly on the set $\{ W \le n \}$. 
\

We now pick any $x,y \in \M, f\in C_0^\infty(\M)$ and $n$ big enough so that $x,y \in \{ W \le n \}$ and $\mathbf{Supp} (f) \subset  \{ W \le n \}$. We have from the previous inequality
\[
| P^n_t f (x) -P^n_t f (y) | \le C d(x,y),
\]
and thus, by taking the limit when $n \to \infty$,
\[
| P_t f (x) -P_t f (y) | \le C d(x,y).
\]
We therefore reach the important conclusion that $P_t$ transforms $C_0^\infty(\M)$ into a subset of the set of smooth and Lipschitz functions. With this conclusion in hands, we can now run the usual Bakry-Emery machinery.

Let $f \in C^\infty_0(\mathbb{M})$,  and $T>0$, and consider the function
\[
\phi (x,t)= \| \nabla P_{T-t} f \|^2(x),
\]
We have
\[
L \phi+\frac{\partial \phi}{\partial t} =2 \mathcal{T}_2(  P_{T-t} f,P_{T-t} f) \ge -2K \phi .
\]
Since we know that $\phi$ is bounded, we can use  a parabolic comparison principle similar  to the one in Theorem \ref{P:missing_key2} to conclude, thanks to Gronwall's inequality,
\[
\| \nabla P_t f \|^2 \le e^{2K t} P_t (\| \nabla f \|^2).
\]
This inequality is then easily extended to any bounded and Lipschitz function $f$.
\end{proof}

Under the same assumptions, we can actually get slightly stronger bounds
\begin{theorem}
Let us assume that for some $K \in \R$,
\[
\mathbf{Ric}_\mathcal{V}  -DY \ge -K,
\]
then for every non negative function $f\in C^\infty(\M)$ such that $\sqrt{f}$ is bounded and Lipschitz, we have for $t \ge 0$
\[
(P_t f) \| \nabla \ln P_t f \|^2 \le e^{2K t} P_t (f \| \nabla \ln f \|^2).
\]
\end{theorem}

\begin{proof}
Notice that if
\[
\phi (x,t)= (P_{T-t} f) \| \nabla \ln P_{T-t} f \|^2(x),
\]
we have
\[
L \phi+\frac{\partial \phi}{\partial t} =2 (P_{T-t} f) \mathcal{T}_2(  \ln P_{T-t} f,\ln P_{T-t} f),
\]
where we use the fact that  since the foliation is totally geodesic we have  for every smooth $g$,
\[
\langle \nabla_\Ho g , \nabla_\Ho \| \nabla_\V g \|^2 \rangle=\langle \nabla_\V g , \nabla_\V \| \nabla_\Ho g \|^2 \rangle.
\]
The proof follows then the same lines as the proof of Theorem \ref{GB18}.
\end{proof}

We now turn to the problem of convergence to an equilibrium for the semigroup $P_t$ and connects this problem to functional inequalities satisfied by the equilibrium measure.  Our first result is the counterpart to Kolmogorov type operators of the famous Bakry-\'Emery criterion \cite{BE}.

\begin{theorem}
Assume that for some $\rho >0$
\[
\mathbf{Ric}_\mathcal{V}  -DY \ge \rho
\]
and that there exists a probability measure $\mu$ on $\M$ such that for every $x \in \M$ and bounded $f$,
\[
\lim_{t \to +\infty} P_t f (x) =\int_\M f d\mu.
\]
Then,  $\mu$  satisfies the log-Sobolev inequality
\[
\int_{\mathbb{M}}f\|  \nabla \ln f \|^2  d\mu \ge \frac{1}{2\rho}  \left[ \int_{\mathbb{M}} f \ln f d\mu -\left( \int_{\mathbb{M}} f d\mu\right)\ln \left( \int_{\mathbb{M}} f d\mu\right) \right].
\]
\end{theorem}

\begin{proof}
Let $g\in C_0^\infty(\M)$, $ g \ge 0$ and denote $f=g+\varepsilon$ where $\varepsilon >0$. Since $\mu$ needs to be an invariant measure for $L$, we have
\begin{align*}
\int_\bM f \ln f d\mu -\int_\bM f d\mu \ln \int_\bM f d\mu &=-\int_0^{+\infty} \frac{\partial}{\partial t} \int_\bM (P_t f) (\ln P_t f) d\mu dt \\
 & =-\int_0^{+\infty}   \int_\bM (L P_t f) ( \ln P_t f) d\mu dt \\
 &=\int_0^{+\infty}   \int_\bM \frac{\| \nabla_\V P_t f\|^2}{P_t f} d\mu dt \\
 &=\int_0^{+\infty}   \int_\bM P_t f  \| \nabla_\V  \ln P_t f \|^2 d\mu dt \\
 &\le \int_0^{+\infty} e^{-2\rho t } dt \int_\bM   \| f \nabla  \ln f \|^2  d\mu \\
 &\le \frac{1}{2\rho} \int_\bM  f  \|  \nabla  \ln f \|^2  d\mu.
\end{align*}
We extend then the inequality to any non negative $f$ such that $\sqrt{f}$ is bounded and Lipschitz.
\end{proof}

We now study the converse question which is to understand how a functional inequality satisfied by an invariant measure implies the convergence to equilibrium of the semigroup.

The easiest convergence to deal with is the $L^2$ convergence and, as it is well-known, is connected to the Poincar\'e inequality.

\begin{theorem}\label{poinca}
Assume that there exist two constants $\rho_1 \ge 0$, $\rho_2 >0$ such that for every $X \in \Gamma^\infty(T\M)$,
\[
\langle (\mathbf{Ric}_\mathcal{V}  -DY )(X) , X \rangle \ge -\rho_1 \| X \|_\V^2+\rho_2 \| X \|_\Ho^2
\]
Assume moreover that  the operator $L$ admits  an invariant  probability measure $\mu$ that satisfies the Poincar\'e inequality
\[
\int_{\mathbb{M}}\| \nabla f \|^2 d\mu \ge \kappa \left[ \int_{\mathbb{M}} f^2 d\mu -\left( \int_{\mathbb{M}} f d\mu\right)^2 \right].
\]
Then, for every bounded and Lipschitz function $f$ such that  $\int_{\mathbb{M}} f d\mu=0$,
\begin{align*}
  (\rho_1+\rho_2)\int_{\mathbb{M}} (P_t f)^2d\mu +\int_{\mathbb{M}} \| \nabla P_t f \|^2d\mu  
 \le   e^{-\lambda t}\left(   (\rho_1+\rho_2) \int_{\mathbb{M}} f^2d\mu + \int_{\mathbb{M}} \| \nabla f \|^2 d\mu \right) ,
\end{align*}
where $\lambda=\frac{2\rho_2 \kappa}{\kappa+\rho_1+\rho_2}$.
\end{theorem}

 \begin{proof}
 We fix $t>0$ and consider the functional
  \[
  \Psi(s) =(\rho_1+\rho_2)P_s( (P_{t-s} f)^2 ) + P_s( \| \nabla P_{t-s}f \|^2 ).
  \]
  By repeating the arguments of the proof of the previous theorem, we get the differential inequality
  \[
  \Psi(s)-\Psi(0) \ge 2 \rho_2 \int_0^s P_u(\| \nabla P_{t-u}f\|^2) )du.
  \]
 Denote now $\varepsilon=\frac{\rho_1 +\rho_2}{\kappa +\rho_1 +\rho_2}$. We have  from the assumed Poincar\'e inequality
\[
\varepsilon\int_{\mathbb{M}} \| \nabla P_{t-u}f\|^2  d\mu \ge \varepsilon \kappa  \int_{\mathbb{M}} (P_{t-u}f)^2 d\mu .
\]
Therefore, denoting $\Theta(s)=\int_{\mathbb{R}^{2n}} \Psi(s)d\mu$, we obtain
\begin{align*}
 \Theta(s)-\Theta(0)&  \ge 2 \eta(1-\varepsilon)\int_0^s  \int_{\mathbb{M}} \|\nabla P_{t-u}f\|^2  d\mu du+2\varepsilon \kappa \int_0^s  \int_{\mathbb{M}} (P_{t-u}f)^2 d\mu du \\
  &\ge \lambda \int_0^s \Theta (u) du.
\end{align*}
We conclude then with Gronwall's differential inequality.
  \end{proof}
  
We can similarly prove a convergence to equilibrium in the entropic distance provided the assumption that the invariant measure satisfies a log-Sobolev inequality.

\begin{theorem}\label{logsob}
Assume that there exist two constants $\rho_1 \ge 0$, $\rho_2 >0$ such that for every $X \in \Gamma^\infty(T\M)$,
\[
\langle (\mathbf{Ric}_\mathcal{V}  -DY )(X) , X \rangle \ge -\rho_1 \| X \|_\V^2+\rho_2 \| X \|_\Ho^2
\]
Assume moreover  that  the operator $L$ admits  an invariant  probability measure $\mu$ that satisfies the log-Sobolev inequality
\[
\int_{\mathbb{M}}f\|  \nabla \ln f \|^2  d\mu \ge \kappa \left[ \int_{\mathbb{M}} f \ln f d\mu -\left( \int_{\mathbb{M}} f d\mu\right)\ln \left( \int_{\mathbb{M}} f d\mu\right) \right].
\]
Then for every positive and bounded $f \in C^\infty(\mathbb{M})$, such that $\| \nabla \sqrt{f} \|$ is bounded and $\int_{\mathbb{M}} f d\mu=1$,
\begin{align*}
 & 2(\rho_1+\rho_2)\int_{\mathbb{M}} P_t f \ln P_t f d\mu +\int_{\mathbb{M}}  P_tf \| \nabla \ln P_tf \|^2 d\mu  \\
 \le &  e^{-\lambda t}\left( 2 (\rho_1+\rho_2) \int_{\mathbb{M}} f \ln f d\mu + \int_{\mathbb{M}}f  \| \nabla \ln f \|^2 d\mu \right) ,
\end{align*}
where $\lambda=\frac{2\rho_2 \kappa}{\kappa+2(\rho_1 +\rho_2)}$.

\end{theorem}

\subsection{The kinetic Fokker-Planck equation}

In this section we study the kinetic Fokker-Planck equation which is an important example of equation to which our methods apply.

Let $V:\mathbb{R}^n \to \mathbb{R}$ be a smooth function. The kinetic Fokker-Planck equation with confinement potential $V$ is the parabolic  partial differential equation:
\begin{equation}\label{FP2}
\frac{\partial h}{\partial t}=\Delta_v h - v \cdot \nabla_v h+\nabla_x V \cdot \nabla_v h -v\cdot \nabla_x h , \quad (x,v) \in \mathbb{R}^{2n}.
\end{equation}
The operator
\[
L=\Delta_v  - v \cdot \nabla_v +\nabla_x V \cdot \nabla_v -v\cdot \nabla_x 
\]
is a Kolmogorov type operator. The foliation on $\R^{2n}$ which is relevant here is not the trivial one. We endow $\R^{2n}$ with the translation invariant metric that makes
\[
\left\{ 2\frac{\partial}{\partial x_i}+\frac{\partial }{\partial v_i}, \frac{\partial }{\partial v_i}, 1 \le i \le n \right\},
\]
an orthonormal basis at any point. We consider then the foliations with leaves $\{ (x,v), v \in \R^n \}$. It is obviously totally geodesic.

\

 The operator $L$ admits for invariant measure the measure
\[
d\mu=e^{-V(x)-\frac{\| v \|^2}{2}} dxdv.
\]
It is readily checked that $L$ is not symmetric with respect to $\mu$. The operator $L$ is the generator of a strongly continuous sub-Markov semigroup $(P_t)_{t \ge 0}$. If we assume that the Hessian $\nabla^2 V$ is bounded, which we do in the sequel, then $P_t$ is Markovian . 

Observe that since $\nabla V$ is Lipschitz, the function $W(x,v)=1+\| x \|^2 +\| v\|^2$ is  such that, for some constant $C>0$, $LW \le C W$ and $\| \nabla W \| \le C W$.

The quadratic form $\mathcal{T}_2$ is easy to compute in this case and we obtain then the following result that was first obtained in \cite{baudoin-bakry}:

 \begin{proposition}\label{lowerbound}
 For every $0 < \eta <\frac{1}{2}$, there exists $K(\eta) \ge -\frac{1}{2}$ such that for every $f \in C^\infty(\mathbb{R}^{2n})$, 
 \[
 \mathcal{T}_2(f,f) \ge -K(\eta) \| \nabla_\V f \|^2+\eta \| \nabla_\Ho f \|^2.
 \]
 \end{proposition}

 The previous lemma shows that Theorems \ref{poinca} and \ref{logsob} thus apply to the kinetic Fokker-Planck operator. We mention that the entropic of the semigroup under the assumption that the invariant measure satisfies a log-Sobolev inequality was first established by Villani  (see Theorem 35 in \cite{Villani1}) but the rate of convergence given by Theorems \ref{poinca} and \ref{logsob} is more explicit.


\begin{thebibliography}{10}
 
 \bibitem{ABGR}
A.  Agrachev, U. Boscain, J.P. Gauthier \& F. Rossi, \emph{The intrinsic hypoelliptic Laplacian and its heat kernel on unimodular Lie groups.} J. Funct. Anal. 256 (2009), no. 8, 2621-2655
  
\bibitem{Bak}
D.  Bakry, 
\emph{Un crit\`ere de non-explosion pour certaines diffusions sur une vari\'et\'e riemannienne compl\`ete}. (French. English summary) [A non-explosion criterium for some diffusions on a complete Riemannian manifold] 
C. R. Acad. Sci. Paris S\'er. I Math. 303 (1986), no. 1, 23-26. 

\bibitem{bakry-stflour}
D. Bakry, \emph{L'hypercontractivit\'e et son utilisation en
th\'eorie des semigroupes}, Ecole d'Et\'e de Probabilites de
St-Flour, Lecture Notes in Math, (1994).

\bibitem{BBBQ} D. Bakry, F. Baudoin,  M. Bonnefont, B. Qian : \emph{Subelliptic Li-Yau estimates on three dimensional model spaces}, Potential Theory and Stochastics in Albac, Aurel Cornea Memorial Volume (2009).
 
 \bibitem{BE} 
D. Bakry \& M. Emery, \emph{Diffusions hypercontractives}, S\'emin. de probabilit\'es XIX, Univ. Strasbourg, Springer, 1983.
 

\bibitem{BL}
D. Bakry \& M. Ledoux:  \emph{Sobolev inequalities and Myers diameter theorem for an abstract Markov generator}. Duke Math. J. 85 (1996), no. 1, 253-270

\bibitem{BL2}
D. Bakry \& M. Ledoux:  \emph{A logarithmic Sobolev form of the Li-Yau parabolic inequality.} Rev. Mat. Iberoam. 22 (2006), no. 2, 683-702.
 
 \bibitem{B}
 F. Baudoin, \emph{Stochastic analysis on sub-Riemannian manifolds with transverse symmetries}, 2014, To appear in Annals of Probability, Arxiv preprint,  http://arxiv.org/abs/1402.4490
 
 \bibitem{baudoin-bakry}
 F. Baudoin, \emph{Bakry-Emery meet Villani}, 2013, Arxiv preprint, http://arxiv.org/abs/1308.4938
  
\bibitem{BB}
F. Baudoin \&  M. Bonnefont, \emph{Log-Sobolev inequalities for subelliptic operators satisfying a generalized curvature dimension inequality}, Journal of Functional Analysis, Volume 262 ~(2012), 2646--2676.

\bibitem{BBG}
F. Baudoin,  M. Bonnefont \& N. Garofalo, \emph{A sub-Riemannian curvature-dimension inequality, volume doubling property and the Poincar\'e inequality}, Math. Ann. 358 (2014), no. 3-4, 833-860.

\bibitem{BBGM}
F. Baudoin,  M. Bonnefont, I. Munive \& N. Garofalo, \emph{Volume and distance comparison theorems for sub-Riemannian manifolds}, To appear in Journal of Functional Analysis (2014), Arxiv preprint, http://arxiv.org/abs/1211.0221

\bibitem{BG0}
F. Baudoin \& N. Garofalo, \emph{Generalized Bochner formulas and Ricci lower bounds for sub-Riemannian manifolds of rank two}, http://arxiv.org/abs/0904.1623

 \bibitem{BG} 
F. Baudoin \& N. Garofalo, \emph{Curvature-dimension inequalities and Ricci lower bounds for sub-Riemannian manifolds with transverse symmetries},To appear in the Journal of the EMS, Arxiv preprint, http://arxiv.org/abs/1101.3590

\bibitem{BG2}
F. Baudoin \& N. Garofalo,
\emph{A note on the boundedness of Riesz transform for some subelliptic operators.} Int. Math. Res. Not. IMRN 2013, no. 2, 398-421. 

\bibitem{BK}
F. Baudoin \& B. Kim, \emph{Sobolev, Poincar\'e and isoperimetric inequalities for subelliptic diffusion operators satisfying a generalized curvature dimension inequality}, Revista Matematica Iberoamericana, 30, (2014), 1, 109-131

\bibitem{BK2}
F. Baudoin \& B. Kim, \emph{The Lichnerowicz-Obata theorem on sub-Riemannian manifolds with transverse symmetries}, To appear in Journal of Geometric Analysis, (2014).

\bibitem{BKW}
F. Baudoin, B. Kim \& J. Wang, \emph{Transverse Weitzenbšck formulas and curvature dimension inequalities on Riemannian foliations with totally geodesic leaves}, 2014, http://arxiv.org/abs/1408.0548

\bibitem{BW1}
F. Baudoin \& J. Wang, \emph{The subelliptic heat kernel on the CR sphere}. Math. Z. 275 (2013), no. 1-2, 135-150

\bibitem{BW2}
F. Baudoin \& J. Wang, \emph{The Subelliptic Heat Kernels of the Quaternionic Hopf Fibration}. Potential Analysis (2014).

\bibitem{BGG} Beals, R., Gaveau, B., Greiner, P. C. \emph{Hamilton-Jacobi theory and the heat kernel on Heisenberg groups,} J. Math. Pures Appl. 79, 7 (2000) 633-689

\bibitem{BeBo}
L. B\'erard-Bergery, J.P. Bourguignon,  \emph{Laplacians and Riemannian submersions with totally geodesic fibres}. Illinois J. Math. 26 (1982), no. 2, 181-200

\bibitem{Besse} 
A. Besse, \emph{Einstein manifolds.}  Reprint of the 1987 edition. Classics in Mathematics. Springer-Verlag, Berlin, 2008. xii+516 pp.

\bibitem{Bonnefont1}
M. Bonnefont, \emph{Functional Inequalities for Subelliptic Heat Kernels}, Phd dissertation, Paul Sabatier University, Toulouse, 2009

\bibitem{Bonnefont2}
M. Bonnefont, \emph{The subelliptic heat kernel on SL(2,R) and on its universal covering: integral representations and some functional inequalities}. Potential analysis. 36 (2012), no. 2, 275-300.

\bibitem{CMV}
D. C.  Chang, I Markina, A. Vasil'ev,  \emph{Hopf fibration: geodesics and distances}. J. Geom. Phys. 61 (2011), no. 6, 986-1000
 
\bibitem{Davies} 
E.B. Davies, \emph{Heat kernels and spectral theory}. Cambridge Tracts in Mathematics, 92. Cambridge University Press, Cambridge, 1989.

\bibitem{Dragomir} 
S. Dragomir \& G. Tomassini, \emph{Differential geometry and analysis on CR manifolds}, Birkh\"auser, Vol. 246, 2006.

\bibitem{EH} J-P. Eckmann, M. Hairer, \emph{Spectral Properties of Hypoelliptic Operators}, Communications in Mathematical Physics, April 2003, Volume 235, Issue 2, pp 233-253

\bibitem{escobales}
R. Escobales, \emph{ Riemannian submersions with totally geodesic fibers. } J. Differential Geom. 10 (1975), 253-276.
 

\bibitem{Elworthy} K.D. Elworthy, \emph{Decompositions of diffusion operators and related couplings}, preprint 2014

\bibitem{Fu}
M. Fukushima, \emph{Dirichlet forms and Markov processes},
Amsterdam-Oxford-New York, North Holland, 1980.



\bibitem{Gaveau} B. Gaveau, \emph{Principe de moindre action, propagation de la chaleur et estim\'ees sous elliptiques sur certains groupes nilpotents}. Acta Math. 139 (1977), no. 1-2, 95-153

\bibitem{Geller} D. Geller, \textit{The Laplacian and the Kohn Laplacian for the sphere}, J. Differential Geometry. \textbf{15} (1980) 417-435

\bibitem{GT1} E. Grong, A. Thalmaier, \emph{Curvature-dimension inequalities on sub-Riemannian manifolds obtained from Riemannian foliations, Part I}, http://arxiv.org/abs/1408.6873

\bibitem{GT2} E. Grong, A. Thalmaier, \emph{Curvature-dimension inequalities on sub-Riemannian manifolds obtained from Riemannian foliations, Part II}, http://arxiv.org/abs/1408.6872

 \bibitem{HN1}B. Helffer, F. Nier: \emph{Hypoelliptic estimates and spectral theory for Fokker-Planck operators and Witten Laplacians.} Lecture Notes in Mathematics, 1862. Springer-Verlag, Berlin, (2005).
 

\bibitem{Hermann} R. Hermann, \emph{A sufficient condition that a mapping of Riemannian manifolds be a fibre bundle},  Proc. Amer. Math. Soc., vol II, (160), pp. 236-242.

\bibitem{Hladky} R. Hladky, \emph{Connections and Curvature in sub-Riemannian geometry}. Houston J. Math, 38 (2012), no. 4, 1107-1134

\bibitem{Kim} B. Kim, \emph{Poincar\'e inequality and the uniqueness of solutions for the heat equation associated with subelliptic diffusion operators}, http://arxiv.org/abs/1305.0508

\bibitem{Li1}
H.-Q. Li, \emph{Estimation optimale du gradient du semi-groupe de la chaleur sur le groupe de Heisenberg}, J. Funct. Anal. 236 (2) (2006) 369-394. 

\bibitem{Li2}
H.-Q. Li, \emph{Estimations asymptotiques du noyau de la chaleur sur les groupes de Heisenberg}, C. R. Acad. Sci. Ser. I Math. (2007) 497-502. 

\bibitem{LY}
P. Li \& S. T. Yau, \emph{On the parabolic kernel of the
Schr\"odinger operator}, Acta Math., \textbf{156}~(1986), 153-201.

\bibitem{MM}
M. G. Molina \& I. Markina, \emph{Sub-Riemannian geodesics and heat operator on odd dimensional spheres}. Anal. Math. Phys. 2 (2012), no. 2, 123-147.


\bibitem{reed1}  M. Reed \& B. Simon, \emph{Methods of modern mathematical physics. Functional analysis}. Second edition. Academic Press, Inc. [Harcourt Brace Jovanovich, Publishers], New York, 1980.

\bibitem{Strichartz} R. Strichartz, \emph{Analysis of the Laplacian on the complete Riemannian manifold}, Journal Func. Anal., 52, 1, 48-79, (1983).


\bibitem{Tanno} Tanno, S., \emph{Variation problems on contact Riemannian manifolds}, Trans. Amer. Math. Soc. (1989), Vol. (1)314, 349-379. 



\bibitem{T2}Taylor, M. E. \textit{Partial differential equations. II,} Applied Mathematical Sciences \textbf{116}, Springer-Verlag, New York (1996) 

\bibitem{Tondeur} P.  Tondeur, \emph{Foliations on Riemannian manifolds}.  Universitext. Springer-Verlag, New York, 1988. xii+247 pp

\bibitem{Villani1} C. Villani: \emph{Hypocoercivity},  Mem. Amer. Math. Soc. 202 (2009), no. 950.

\bibitem{Vilms} J. Vilms, \emph{Totally geodesic maps}, Journal of differential geometry, vol. 4, (1970), 73-79

\bibitem{FYW2}F-Y Wang: \emph{Generalized Curvature Condition for Subelliptic Diffusion Processes}, http://arxiv.org/pdf/1202.0778v2

\bibitem{FYW3}F-Y Wang: \emph{Analysis for diffusion processes on Riemannian manifolds}, Advanced Series on Statistical Science and Applied Probability, Vol. 18. World Scientific, (2014).

\bibitem{Wang1}
J. Wang: \emph{Sub-Riemannian heat kernels on model spaces and curvature-dimension inequalities on contact manifolds}, Phd dissertation, Purdue University, (2014)

\bibitem{Wang2}
J. Wang: \emph{The Subelliptic Heat Kernel on the CR hyperbolic spaces}, Preprint, 2014

\bibitem{Wu}
L. Wu, \emph{Large and moderate deviations and exponential convergence for stochastic damping Hamiltonian systems}. Stochastic Process. Appl. 91 (2001), no. 2, 205Ð238



\end{thebibliography}
\end{document}